\documentclass[11pt,a4paper,reqno]
{amsart}
\usepackage[utf8]{inputenc}
\usepackage[T1]{fontenc}
\usepackage[francais,english]{babel}

\usepackage{amsmath,xcolor,color,dsfont,verbatim}
\usepackage{amsfonts}
\usepackage{amssymb,amsthm}

\usepackage{mathrsfs}

\usepackage{graphicx}

\usepackage{subfigure}

\usepackage{lscape}

\usepackage[normalem]{ ulem }
\usepackage{soul}

\colorlet{darkgreen}{green!50!black}

\usepackage{hyperref} 
\hypersetup{	
colorlinks=true,  
breaklinks=true,  
urlcolor= green, 
linkcolor= blue,	
citecolor=darkgreen,	
pdftitle={Green's functions with oblique Neunann boundary conditions in the quadrant}, 
pdfauthor={Sandro Franceschi},
}

\newtheorem{thm}{Theorem}

\newtheorem{prop}[thm]{Proposition}

\newtheorem{deff}[thm]{Definition}

\newtheorem{lem}[thm]{Lemma}

\theoremstyle{definition}
\newtheorem{rem}[thm]{Remark}

\newtheorem*{remarque*}{Remarque}
\newtheorem*{remark*}{Remark}

\usepackage{epsfig}
\usepackage{dsfont}

\usepackage{geometry}
\usepackage{enumitem}
\usepackage{mathtools}
\usepackage{eqnarray,cases}
\usepackage{multirow}

\usepackage{indentfirst}

\usepackage{pstricks-add}
\usepackage[numbers]{natbib}

\newcommand{\citeep}[1]{\citeeauthor*{#1} (\citeeyear{#1}) \citeep{#1}}

\usepackage{float}

\newcommand{\G}{\mathcal G_\mathcal R}
\newcommand{\R}{\mathcal{R}}

\let\phi=\varphi
\let\epsilon=\varepsilon
\newcommand{\rems}[1]{\textcolor{black}{#1}}
\newcommand{\remst}[1]{}
\newcommand{\remstf}[1]{}

\title
[Green's functions with oblique Neumann boundary conditions in the quadrant]
{Green's functions with oblique Neumann boundary conditions in \rems{the quadrant}}

\author{S.\ Franceschi}

\address{Laboratoire de Probabilités, Statistique et Modélisation, Sorbonne Universit\'e,
        4 Place Jussieu, 75252 Paris Cedex 05, France
        } \email{Sandro.Franceschi@upmc.fr}


\begin{document}

\maketitle


\begin{abstract}
We study semi-martingale obliquely reflected Brownian motion (SRBM) with drift in \rems{the first quadrant} of the plane in the transient case. 
Our main result determines 
a general explicit integral expression for the 
\rems{moment generating function}
of Green's functions of this process.
To that purpose we establish 
a new kernel functional equation connecting 
\rems{moment generating functions}
of Green's functions inside the \rems{quadrant} and on its edges.
This is reminiscent of the recurrent case where a functional equation derives from the basic adjoint relationship which characterizes the stationary distribution.
This equation 
leads us to a 
non-homogeneous Carleman boundary value problem. Its resolution 
provides
a 
formula 
for
the 
\rems{moment generating function}
in terms of contour integrals and a conformal mapping.
\end{abstract}

\footnote{Version of \today}



\section{Introduction}
\label{sec:introduction}

\subsection{Overview}
\label{subsec:overview}

\subsubsection*{Main goal}
In this article, we consider ${\rems{Z=(Z(t), t \geqslant 0)}}$, an \textit{obliquely reflected Brownian motion with drift} in $\mathbb{R}_+^2$ starting from $x$. \rems{Denote the} transition \rems{semigroup by $(P_t)_{t\geqslant 0}$}. We will focus on the quadrant case because thanks to a simple linear transform it is easy to extend all the results to any wedge, see \cite[Appendix A]{franceschi_explicit_2017}. 
\rems{This process behaves as a Brownian motion with drift vector $\mu$ and covariance matrix $\Sigma$ in the interior of this quadrant} and reflects
instantaneously in a constant direction 
$R_i$ for $i=1,2$ on each edge, see Figure~\ref{fig:rebondderive} and Proposition~\ref{def:MBsemimartingale} for more details. We are \rems{interested} in the case where this process is transient, that is when the parameters make the process \rems{tend} to infinity almost surely, see Section \ref{subsec:recurrence}.

\begin{figure}[H]
\center{
\includegraphics[scale=0.5]{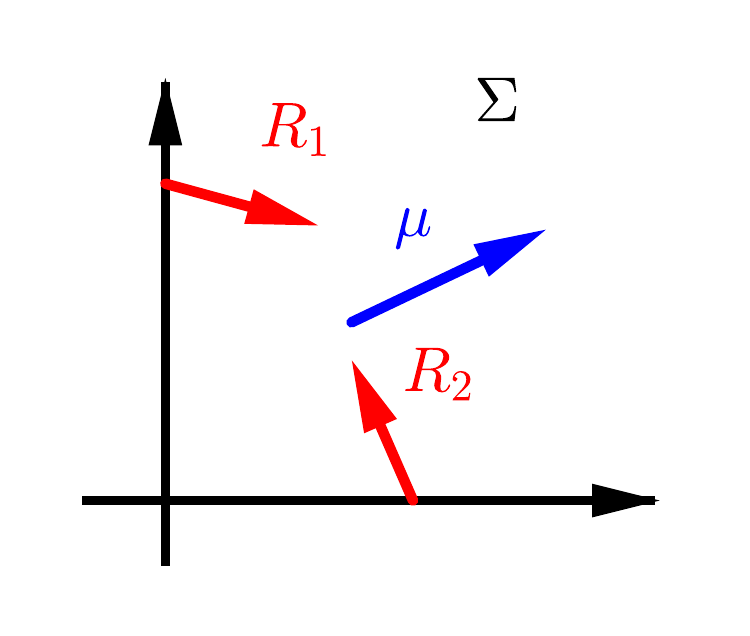}
}
\caption{Reflection vectors and drift
}\label{fig:rebondderive}
\end{figure}

The main goal of this article 
is to study $G$, 
\rems{\textit{Green's measure}} (\textit{potential kernel}) of ${\rems{Z}}$:
$$
G(x, A) := \mathbb{E}_{x} \left[ \int_{0}^{\infty} \mathbf{1}_A ({\rems{Z(t)}}) \, \mathrm{d} t \right]
= \int_0^{\infty} P_t(x ,A)\, \mathrm{d} t
$$
which represents the mean time spent by the process in some measurable set $A$
of the quadrant. Let us remark that if $A$ is bounded and if ${\rems{Z}}$ is transient then $G(x,A)$ is finite. 
The density of the measure $G$ \rems{with respect} to the Lebesgue measure is called\remst{ the} \textit{Green's function} and is equal to
$$g(x,\cdot) :=\int_0^{\infty} p_t(x,\cdot)\, \mathrm{d} t,$$ if we assume that $p_t$ is a transition density for ${\rems{Z(t)}}$.
The kernel $G$ defines a potential operator
$$
G f(x):
=\mathbb{E}_{x} \left[ \int_0^{\infty}  f({\rems{Z(t)}}) \, \mathrm{d} t \right] 
=
\int_{\mathbb{R}_+^2}  f(y) \ g(x,y) \,  \mathrm{d} y
,
$$
for every positive measurable function $f$.
We define $H_1$ and $H_2$ the \textit{boundary Green's measures} on the edges
such that for $i=1,2$,
$$
H_i(x,A)
:= 
\mathbb{E}_{x} \left[ \int_0^\infty {\mathbf{1}}_{A} ({\rems{Z(t)}})\,  \mathrm{d}{\rems{L_i(t)}} \right] 
$$
where we integrate with respect to ${\rems{L_i(t)}}$, 
the local time of the process on the edge $z_i=0$.
The support of $H_1$ lies on the \rems{vertical axis} 
and the \rems{support} of $H_2$ lies on the 
\rems{horizontal axis}. We can say that $H_i (x,A)$ represents the mean local time spent on the corresponding edge. 
When it exists, the density of the measure $H_i$ \rems{with respect} to the Lebesgue measure is denoted $h_i$
and the boundary potential kernel is given by
 \begin{equation*}
\label{intro:eq:greenbordformule}
H_i f (x):=
\mathbb{E}_{x} \left[\int_0^\infty f({\rems{Z(t)}}) \, \mathrm{d}L^i_t \right] = \int_{\mathbb{R}_+^2} f(y) h_i(x,y) \, \mathrm{d}y.
\end{equation*}
In this article we determine an explicit formula for $\psi^x$ and $\psi^x_i $ the Laplace transforms
of $g$ and $h_i$ \rems{usually named \textit{moment generating
functions}}, defined by
\begin{equation}
\label{eq:defpsi}
\psi^x (\theta)
:=\mathbb{E}_{x} \left[ \int_0^{\infty}  e^{\theta \cdot  {\rems{Z(t)}}} \, \mathrm{d} t \right]
\text{ and }
\psi_i^x (\theta):= 
\mathbb{E}_{x} \left[ \int_0^{\infty}  e^{ \theta \cdot  {\rems{Z(t)}} }  \,  \mathrm{d} {\rems{L_i(t)}}  \right]
\end{equation}
where $\theta=(\theta_1,\theta_2)\in\mathbb{C}^2$. 
Thereafter we will often omit to write the $x$. Furthermore we notice that the functions $\psi_i$ depend on only one variable. We will then denote them by $\psi(\theta)$, $\psi_1(\theta_2)$ and $\psi_2(\theta_1)$.

\subsubsection*{Context}

Obliquely reflected Brownian motion in the quadrant and in orthants of any dimensions was introduced and extensively studied in the eighties by Harrison, Reiman, Varadhan and Williams \cite{HaRe-81,HaRe-81b,
varadhan_brownian_1985,
williams_recurrence_1985,
Williams-85}.
The initial motivation for the study of this kind of processes was because it serves as an approximation of large queuing networks as we can see in
\cite{Foschini,reiman_84_open,
foddy_analysis_1984,
baccelli_analysis_1987,harrison_brownian_1987}.
Recurrence or transience in two dimensions, which is an important aspect for us, was studied in \cite{williams_recurrence_1985,
dai_steady-state_1990,
hobson_recurrence_1993}.
In higher dimensions the problem is more complex, see for example 
\cite{Ch-96,
BrDaHa-10,
Br-11,
DaHa-12}.\remst{ Otherwise} The intertwining relations of obliquely reflecting Brownian motion have been studied in \cite{dubedat_reflected_2004,kager_reflected_2005},
its Lyapunov functions in \cite{dupuis_lyapunov_1994},
its cone points in
\cite{lega-87} and its existence in non-smooth planar domains and its links with complex and harmonic analysis in 
\cite{burdzy_obliquely_2017}. \rems{Some articles link SRBM in the orthant to financial models as in \cite{banner2005,ichiba2011}.}
Such a process \rems{and these financial models are} also related to competing Brownian particle systems as in \cite{sarantsev_infinite_2017,
bruggeman_Sarantsev_multiple_2018}.
Finally, some other related stochastic processes have been studied too as two-dimensional oblique Bessel processes in \cite{Lep} and two-dimensional obliquely sticky Brownian motion in \cite{dai_stationary_sticky_2018}.


\subsubsection*{Green's functions and invariant measures}

Green's functions and invariant measures are two similar concepts, the first \rems{dealing with} the transient case and the second the  recurrent case.
Indeed, in the transient case the process spends a finite time in a bounded set while in the recurrent case it spends an infinite time in it.
Thus Green's measure may be interpreted as the average time spent in some set while ergodic theorems say that the invariant measure is the average \textit{proportion} of time spent in some set. 


In the discrete setting, Green's functions of random walks in the quadrant have been studied in several articles, as in the reflecting case in \cite{kurkova_martin_1998} or in the absorbed case in \cite{kourkova_random_2011}. 
To our knowledge it seems that in the continuous setting, Green's functions of reflected Brownian motion in cones has not been studied yet (except in dimension one, see \cite{Chen_basic_1999}).

On the other hand, 
the invariant measure of this kind of 
processes 
has been deeply studied in the literature: 
the asymptotics of the stationary distribution 
is the subject of many articles as \cite{harrison_reflected_2009,
 dai_reflecting_2011,
 dai_stationary_2013,
 franceschi_asymptotic_2016,Sa-+1}, 
numerical methods to compute the stationary distribution have been developed in \cite{dai_steady-state_1990,dai_reflected_1992}
and explicit expressions for the stationary distribution are found in some particular cases in
\cite{foddy_analysis_1984,Foschini,
baccelli_analysis_1987,
harrison_multidimensional_1987,
dieker_reflected_2009,
franceschi_tuttes_2016,
BoElFrHaRa_algebraic_2018} 
and in the general case in
\cite{franceschi_explicit_2017}.

%
%




\subsubsection*{Oblique Neumann boundary problem}
Green's functions and invariant measures of Markov processes are 
central in 
potential theory and in ergodic theorems for additive functionals. 
In particular they give a probabilistic interpretation to the solution\rems{s} of some partial differential equations. Appendix \ref{appendix:potentialtheory} illustrates this. 
Our case is especially complicated because we consider a non-smooth 
unbounded domain, \rems{and reflection at the boundary is oblique.}

\rems{Consider ${\rems{Z}}$, an obliquely reflected Brownian motion with drift vector $\mu$, covariance matrix $\Sigma$, and reflection matrix $R$. Its first and second columns $R^1$ and $R^2$ form reflection vectors at the faces $\{(0,z) | z\geqslant 0 \}$ and $\{(z,0) | z\geqslant 0 \}$.} Its generator inside the quarter plane $\mathcal{L}$ and its dual generator $\mathcal{L}^*$ are equals to
\begin{equation}
\label{eq:def_L_generateur}
\mathcal{L}f=\frac{1}{2} \nabla \cdot \Sigma \nabla + \mu \cdot \nabla 
\quad \text{and} \quad
\mathcal{L}^* f=\frac{1}{2} \nabla \cdot \Sigma \nabla - \mu \cdot \nabla .
\end{equation}
\citet[(8.2) and (8.3)]{HaRe-81}
derive (informally) the \textit{backward} and the \textit{forward equations} (with boundary and initial conditions) for $p_t(x,y)$, the transition density of the process.
The forward equation (or Fokker-Planck equation) may be written as
$$
\begin{cases}
\mathcal{L}^*_y p_t(x,y) = \partial_t p_t(x,y),
\\
\partial_{R_i^*} p_t(x,y ) - 2\mu_i p_t(x,y ) = 0 \text{ if } y_i =0,
\\
p_0(x,\cdot)=\delta_x,
\end{cases}
$$
where
$$
R^* =2\Sigma-R \ \text{diag}(R)^{-1} \text{diag}( \Sigma)\rems{,}
$$
\remst{and denoting }$R^*_i$ \rems{is} its $i$th column and $\partial_{R_i^*} = R_i^* \cdot \nabla_y$ the derivative along $R^*_i$ on the boundary.
\rems{(In 
\cite{HaRe-81}, notation is different: Row vectors instead of column vectors.)}
Letting $t$ going to infinity in the forward equation, Harrison and Reiman conclude that, in the positive recurrent case, \rems{the density $\pi$ of the stationary distribution} satisfies the following steady-state equation \cite[(8.5)]{HaRe-81}
$$
\begin{cases}
\mathcal{L}^* \pi  =0,
\\
\partial_{R_i^*} \pi - 2\mu_i \pi = 0 \text{ if } y_i =0.
\end{cases}
$$
In the transient case, integrating the forward equation in time from $0$ to infinity suggests that the Green's function $g$ satisfies 
the following \rems{partial differential equation with \textit{Robin boundary condition} (specification of the values of a linear combination of a function and its derivative on the boundary)}
\begin{equation}
\label{eq:robineq}
\begin{cases}
\mathcal{L}^*_y g(x,\cdot) = - \delta_x,
\\
\partial_{R_i^*} g(x,\cdot) - 2\mu_i g(x,\cdot) = 0 \text{ if } y_i =0.
\end{cases}
\end{equation}
A similar equation holds in dimension one, see \eqref{eq:PDEdim1}. The Green's function $g$ of the obliquely reflected Brownian motion in the quadrant is then a fundamental solution of the dual operator $\mathcal{L}^*$. Together with the boundary Green's functions $h_i$ they should allow to solve the following oblique Neumann boundary problem
$$
\begin{cases}
\mathcal{L} u = - f & \text{in } \mathbb{R}_+^2,
\\
\partial_{R_i} u = \phi_i & \text{if } y_i =0,
\end{cases}
$$
where $\partial_{R_i} = R_i \cdot \nabla_y$ is the derivative along $R_i$. 
If a solution $u$ exists, it should satisfy 
$$u =Gf +H_1\phi_1 +H_2\phi_2.$$
One may see Appendix \ref{appendix:potentialtheory} to better understand this thought.

\subsection{Main results and strategy}
%
\subsubsection*{Functional equation}
To find $\psi$ and $\psi_i$ \rems{the moment generation functions} of Green's functions, we will establish in Proposition \ref{propeqfoncgreen} a new kernel functional equation connecting what happens inside the quadrant and on its boundaries, namely
\begin{equation}
\label{eq:eqfunctintro}
- \gamma (\theta) \psi (\theta) = \gamma_1 (\theta) \psi_1 (\theta_2) + \gamma_2 (\theta) \psi_2 (\theta_1) + e^{\theta \cdot x}
\end{equation}
where $x$ is the starting point and \rems{the \textit{kernel}} $\gamma$ and $\gamma_i$ are some polynomials given in equation~\eqref{intro:def:gamma}. To our knowledge this formula \rems{has not yet appeared} in the literature. Such an equation is reminiscent of the balance equation satisfied by \rems{the moment generation function} of the invariant measure in the recurrent case which derives from the \textit{basic adjoint relationship}, see \cite[(2.3) and (4.1)]{dai_reflecting_2011} \rems{and \cite[(5)]{franceschi_explicit_2017}}.  The additional term $e^{\theta \cdot x}$ depending on the starting point makes this equation differ from the one of the recurrent case. 
It reminds also of the several kernel equations obtained in the discrete setting in order to study random walks and count walks in the quadrant \cite{fayolle_random_2017,kourkova_functions_2012}. 


\subsubsection*{Analytic approach}
In the seventies, \citet{malysev_analytic_1972,
fayolle_two_1979} introduced an analytic approach to solve 
such functional equations. This method is presented in the famous 
book of \citet{fayolle_random_2017}. Since then, it has been used a lot in the discrete setting in order to solve many problems as counting walks, studying Martin boundaries, determining invariant measures or Green's functions, see
\cite{kurkova_martin_1998,
kourkova_random_2011,
kourkova_functions_2012,
bousquet-melou_walks_2010,
bernardi_counting_2015}.
This approach has also been used in the continuous setting in order to study stationary distributions in a few articles as in \cite{foddy_analysis_1984,Foschini,
baccelli_analysis_1987,
franceschi_explicit_2017}. 
However, to our knowledge it is 
the first time that this method is used to find Green's functions in the continuous case.
To obtain an explicit expression of the Laplace transforms 
using this analytic approach we will go through the following steps:
\begin{enumerate}[label={\rm (\roman{*})},ref={\rm (\roman{*})}]
     \item\label{enumi:functional_equation} Find a functional equation, see Section \ref{subsec:functional_equation};
     \item\label{enumi:Riemann} 
     Study the kernel (and its related Riemann surface) and extend meromorphically the Laplace transforms, 
     see Sections \ref{subsec:noyau} and \ref{subsec:continuation};
     \item\label{enumi:BVP_1} Deduce from the functional equation a boundary value problem (BVP), see Section \ref{subsec:carlemanBVP};
     \item\label{enumi:BVP_2}  Find some conformal \rems{glueing} function and solve the BVP, see Sections \ref{subsec:collage} and \ref{subsec:resolution_BVP}.
\end{enumerate}
For some analytic steps, our strategy of proof is similar to the one used in \cite{franceschi_explicit_2017} to determine the stationary distribution. In some places, the technical details will be identical to \cite{franceschi_explicit_2017}
and \cite{baccelli_analysis_1987},
\rems{especially related to the kernel.}
But, being in the transient case, the probabilistic study differs from \cite{franceschi_explicit_2017} and leads to a different functional equation and to a more complicated boundary value problem whose analytic resolution is more difficult.

 \subsubsection*{Boundary value problem}
 In Lemma \ref{lem:BVP} we establish a Carleman boundary value problem satisfied by the Laplace transform $\psi_1$. For some functions $G$ and $g$ defined in \eqref{eq:def:G} and \eqref{eq:def:g} and some hyperbola $\R$ defined in \eqref{eq:curve_definition1} which depend on the parameters $(\mu,\Sigma,R)$ we obtain the boundary condition \eqref{eq:boundary_condition_general1}:
 $$
     \psi_1(\overline{\theta_2})=G(\theta_2)\psi_1({\theta_2}) + g(\theta_2), \qquad \forall \theta_2\in \R.
 $$
\rems{This equation is particularly complicated: The function g makes the BVP}
doubly non-homogeneous due to the function $G$ but also to $g$ which comes from the term $e^{\theta \cdot z_0}$ in the functional equation \eqref{eq:eqfunctintro}. The function $g$ makes the BVP differ from the one obtained in the recurrent case for the Laplace transform of the stationary distribution \rems{\cite[(22)]{franceschi_explicit_2017}}.

 \subsubsection*{Explicit expression}

The resolution of such a BVP is technical and uses the general theory of BVP. In order to make the paper self-contained, Appendix \ref{appendix:BVP} briefly presents this theory. The solutions can be expressed in terms of Cauchy integral and some conformal mapping $w$ defined in \eqref{eq:expression_CGF_BM1}.
Our main result is an integral formula for the Laplace transform $\psi_1$ precisely stated in Theorem~\ref{thm:main}.
Let us give now the shape of the solution.
We have 
$$
\psi_1(\theta_2)=
\frac{-Y (w(\theta_2))}{2i\pi} \int_{\R^-} \frac{g(t)}{Y^+ (w(t))}
\left(
 \frac{w'(t) }{w(t)-w(\theta_2)}
 + \chi
  \frac{ w'(t) }{w(t)} 
  \right)
  \, \mathrm{d}t
$$
where 
$$
Y (w(\theta_2))=w(\theta_2)^{\chi} 
\exp  \left(
\frac{1}{2i\pi} \int_{\R^-} \log(G(s))\left(\frac{ w'(s)}{w(s)-w(\theta_2)}- \frac{ w'(s)}{w(s)} \right) \, \mathrm{d}s
\right),
$$
$\chi=0$ or $1$ and $Y^+$ is the limit of $Y$ on $\R$. \rems{This formula is analogous but more complicated than the one obtained in \cite[(14)]{franceschi_explicit_2017}.}
In the same way there is a similar formula for $\psi_2$, and then the functional equation \eqref{eq:eqfunctintro} gives an explicit formula for the Laplace transform $\psi$. \rems{Green's functions} are obtained by taking the inverse Laplace transforms.


\subsection{Perspectives}

Developing the analytic approach, it would be certainly be possible to study further 
Green's functions and obliquely reflected Brownian motion in wedges. Here are some research topic perspectives:
\begin{itemize}
\item
Study the algebraic nature of \remst{the }Green's function: as in the discrete models it would require to introduce the group related to the process and analyze further the structure of the BVP in studying the existence of multiplicative and additive decoupling functions, see Section \ref{subsec:decoupling_functions} and \cite{malysev_analytic_1972,bousquet-melou_walks_2010,
bernardi_counting_2015,
BoElFrHaRa_algebraic_2018};
\item
Determine the asymptotics of \remst{the }Green's function, the Martin boundary and the corresponding harmonic functions: to do this we should study the singularities and \rems{invert} the Laplace transforms in order to use transfer lemmas and the saddle point method on the Riemann surface, see \cite{kurkova_martin_1998,
kourkova_random_2011, franceschi_asymptotic_2016,
raschel_random_2014,
martin_boundary_raschel_tarrago,ernst_franceschi_asymptotic};
\item Give an explicit expression for the transition function: to do that, we could try to find a functional equation satisfied by the resolvent of the process, which would contain one more variable, and seek to solve it.
\end{itemize}
We leave these questions for future works. Furthermore, even if there are some attempts, extending the analytic approach to higher dimensions 
remains an open question.

\subsection{Structure of the paper}

\begin{itemize}
\item Section \ref{sec:transientSRBM} presents the process we are studying 
\rems{and focuses on the transience conditions.}
\item Section \ref{sec:functionalequation} establishes the new functional equation which is the starting point of our analytic study. The kernel is studied and the Laplace transform is continued on some domain.
\item Section \ref{sec:BVP} \rems{states} and solves the boundary value problem satisfied by the Laplace transform $\psi_1$. The main result, which is the explicit expression of $\psi_1$, is stated in Theorem~\ref{thm:main}.
\item Appendix \ref{appendix:potentialtheory} presents in a brief way the potential theory which links Green's functions and the partial differential equations.
\item Appendix \ref{appendix:BVP} presents the general theory of boundary value problems which is used in Section \ref{sec:BVP}.
\item Appendix \ref{appendix:dim1} studies \remst{the }Green's functions of reflected Brownian motion in dimension one.
\item \rems{Appendix \ref{appendix:generalization} explain how to generalize the results to the case of a non-positive drift.}
\end{itemize}

\subsection*{Acknowledgment}

I would like to express my gratitude to
Irina Kourkova and Kilian Raschel for introducing me to this subject and this theory.
This research was partially supported by the ERC starting grant - 2018/2022 - COMBINEPIC - 759702.

 \section{Transient SRBM in the quadrant}
\label{sec:transientSRBM}

\subsection{Definition}
\label{intro:subsec:defexis}



Let
\begin{equation*}
\label{intro:eq:defmuRsigma}
\Sigma = \left(  \begin{array}{cc} \sigma_{11} & \sigma_{12} \\ \sigma_{12} & \sigma_{22} \end{array} \right) \in \mathbb{R}^{2 \times 2},
\quad
\mu= \left(  \begin{array}{c} \mu_1 \\  \mu_2  \end{array} \right) \in \mathbb{R}^2,
\quad
R=(R_1,R_2)= \left(  \begin{array}{cc} 1 & r_{12} \\ r_{21} & 1 \end{array} \right) \in \mathbb{R}^{2 \times 2}
\end{equation*}
respectively be a positive-definite covariance matrix, a drift and a \textit{reflection matrix}.
The matrix $R$ has two reflection vectors giving the reflection direction 
$R_1
$ 
along the $y$-axis and 
$R_2
$ along the $x$-axis, see Figure \ref{fig:rebondderive}.
We will define the obliquely reflected Brownian motion in the quadrant in the case where the process is a semi-martingale, see \citet{Williams-85}. Such a process is also called \textit{semimartingale reflected Brownian motion} (SRBM).
\begin{prop}[Existence and uniqueness]
\label{def:MBsemimartingale}
Let us define ${\rems{Z=(Z(t),t\geqslant 0)}}$ a SRBM with drift in the quarter plane $\mathbb{R}_+^2$ associated to $(\Sigma, \mu, R)$ as the semi-martingale such that for $t\in \mathbb{R}_+$ we have
$$ {\rems{Z(t)}}=x + W\rems{(t)} + \mu t + R L\rems{(t)} \ \in \mathbb{R}_+^2 ,$$ 
where
$x$ is the starting point, $W$ is a planar Brownian motion starting from $0$ and of covariance $\Sigma$ and for $i=1,2$ the coordinate ${\rems{L_i(t)}}$ \rems{of $L\rems{(t)}$} is a continuous non-decreasing process which increases only when $\rems{Z_i}=0$, that is when the process \rems{reaches the face $i$ of the} boundary ($\int_{\{t : {\rems{Z_i(t)}} > 0 \}} \mathrm{d} {\rems{L_i(t)}}=0$).
The process ${\rems{Z}}$
exists \rems{in a weak sense} if and only if \rems{one of the three conditions holds}
\begin{equation}
\label{eq:existence}
 \rems{r_{12}>0,\quad r_{21}>0, \quad r_{12}r_{21}<1.}
\end{equation}
In this case the process is unique in law and defines a Feller continuous strong Markov process.
\end{prop}
The process $L\rems{(t)}$ represents the local time on the boundaries, more specifically its first coordinate $\rems{L_1(t)}$ is the local time on the \rems{vertical axis} and the second coordinate $\rems{L_2(t)}$ the local time on the \rems{horizontal axis}.
The proof of existence and uniqueness can be found in the survey of \citet[Theorem 2.3]{williams_semimartingale_1995} for orthants, \rems{in general dimension $d\geqslant 2$.}
These conditions mean that the reflection vectors must not be too much inclined toward $0$ for the process to exist. Otherwise the process will be trapped in the corner, see Figure \ref{fig:conditionexistence}. \rems{The limit condition $r_{12}r_{21}=1$ is satisfied when the two reflection vectors are collinear and of opposite directions.}

\begin{figure}[htb]
    \centering
    \subfigure[Process doesn't exist]{\label{sub1}  \includegraphics[scale=0.55]{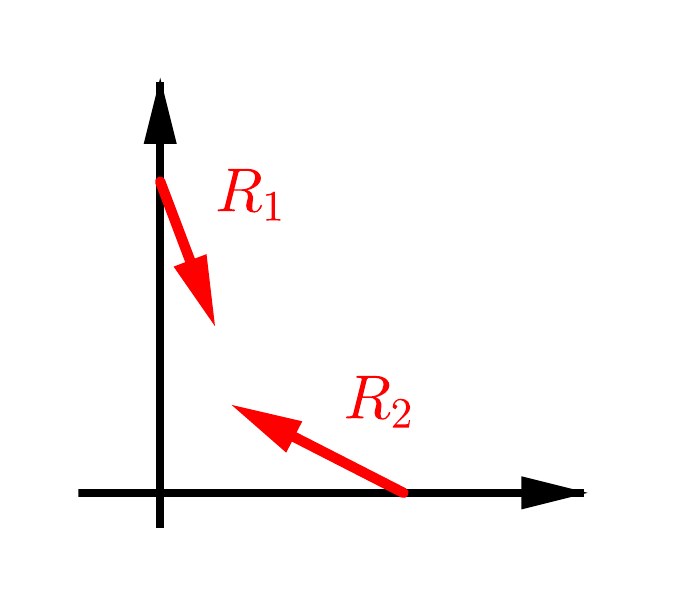}}
    \subfigure[Process exists]{\label{sub2} \includegraphics[scale=0.55]{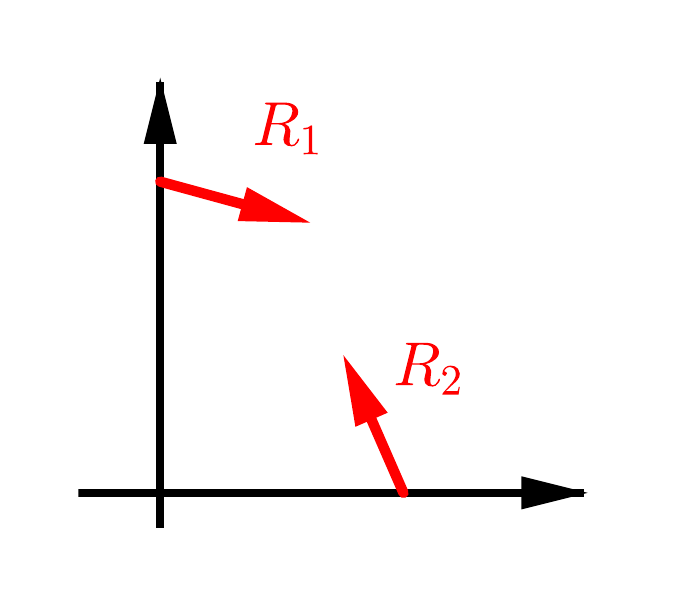}}
    \caption{Existence conditions}
    \label{fig:conditionexistence}
\end{figure}

\subsection{Recurrence and transience}
\label{subsec:recurrence}

Markov processes have approximately two possible behaviours as explained in the book of \citet[p. 424]{revuzyor}. Either
they converge to infinity which is
the transient case, or they come back at arbitrarily large times to some
small sets which is the recurrent case. We present very briefly some results of the corresponding theory, for more details (in particular on topological issues) one can read the articles of
\citet{Azema1966,Azema1967}.

Let ${\rems{X(t)}}$ be a Feller continuous strong Markov process \rems{on} state space $E$, a locally compact set with countable base.
We say that the point $x$ leads to $y$ if for all neighbourhood $V$ of $y$ we have $\mathbb{P}_x (
\tau_V
<\infty)>0$
where $\tau_V=\inf \{t > 0 : {\rems{X(t)}} \in V \}$. The points $x$ and $y$ communicate if $x$ leads to $y$ and $y$ leads to $x$, it is an equivalence relation. For $x\in E$ we say that
\begin{itemize}
\item $x$ is \textit{recurrent} if 
$\mathbb{P}_x \left( \overline{\underset{t\to\infty}{\lim}}
\mathbf{1}_U ({\rems{X(t)}})=1 \right)
=1$ for all $U$ neighbourhoods of $x$,
\item $x$ is \textit{transient} if
$\mathbb{P}_x \left( \overline{\underset{t\to\infty}{\lim}}
\mathbf{1}_U ({\rems{X(t)}})=1 \right)
=0$ for all $U$ \rems{relatively compact neighborhoods} of~$x$.
\end{itemize}
\rems{Each point is either recurrent or transient, and if two states communicate, they are
either both recurrent or both transient, see \cite[Theorem III 1.]{Azema1966}. The process is called \textit{recurrent}
or \textit{transient} if each point is recurrent or transient, respectively.}
The next proposition may be found in \cite[Prop III 1.]{Azema1966}.
\begin{prop}[Transience properties] 
The following properties are equivalent
\begin{enumerate}
\item every point is transient;
\item ${\rems{X(t)}}$ \rems{tends} to infinity when $t\to\infty$ a.s.;
\item for all compact $K$ of $E$ and for all starting point $x$ \remst{the }Green's measure of $K$ is finite:
$$
G(x,K) = \mathbb{E}_{x} \left[ \int_{0}^{\infty} \mathbf{1}_K ({\rems{X(t)}}) \, \mathrm{d} t \right] < \infty .
$$
\end{enumerate}
\end{prop}

The main articles which study the recurrence and the transience of SRBM in wedges are \cite{williams_recurrence_1985} with zero drift,
\cite{hobson_recurrence_1993} with non-zero drift and the survey \cite{williams_semimartingale_1995}.
The process has only one equivalence class equal to the whole quadrant, see for example \cite[(4.1)]{williams_recurrence_1985}.
The process will be recurrent if for each set $V$ (of positive Lebesgue measure) and all starting point $x$, $\mathbb{P}_x ( \tau_V < \infty)=1$, otherwise it will be transient. It will be called \rems{\textit{positive recurrent}} and will admit a stationary distribution if $\mathbb{E}_x [\tau_V] < \infty$ and \rems{\textit{null recurrent}} if $\mathbb{E}_x [\tau_V] = \infty$ for all $x$ and $V$.
\begin{prop}[Transience and recurrence]
Assume that the existence condition \eqref{eq:existence} is satisfied and note $\mu_1^-$ and $\mu_2^-$ the negative parts of the drift components. The process ${\rems{Z}}$ is 
transient if and only if
\begin{equation} 
\mu_1 + r_{12}  \mu_2^- > 0
\text{ or }
\mu_2 + r_{21}  \mu_1^- > 0 ,
\label{eq:transient_condition}
\end{equation}
and recurrent if and only if
\begin{equation} 
\mu_1 + r_{12}  \mu_2^- \leqslant 0
\text{ and }
\mu_2 + r_{21}  \mu_1^- \leqslant 0 .
\end{equation}
In the latter case the process is \rems{positive
recurrent} and admit a unique stationary distribution if and only if
$
\mu_1 + r_{12}  \mu_2^- < 0
\text{ and }
\mu_2 + r_{21}  \mu_1^- < 0 ,
$
and is \rems{null recurrent} if and only if
$
\mu_1 + r_{12}  \mu_2^- = 0
\text{ or }
\mu_2 + r_{21}  \mu_1^- = 0 . 
$
\end{prop}
This result may be found in \cite{williams_semimartingale_1995,hobson_recurrence_1993,
williams_recurrence_1985}. In order to restrict the number of cases to handle, 
we will now assume that the drift has positive coordinates, that is
\begin{equation}
\label{eq:drift_positive_condition}
\mu_1 >0
\text{ and }
\mu_2 >0.
\end{equation}
In this case the process is then obviously transient and converges to infinity. In the other transient cases the process \rems{tends} to infinity but\remst{ only} along one \rems{of the} axis. \rems{See for example \cite{fomichov2020probability} which computes the probability of escaping along each axis when $\mu_1<0$ and $\mu_2<0$.} These cases could be treated \rems{in} the same way with additional technical issues.
\rems{See Appendix~\ref{appendix:generalization} which details the main differences of the study and generalize the results to the case of a non-positive drift.}
Assumption~\eqref{eq:drift_positive_condition} is the counterpart to the rather standard hypothesis made in the recurrent case (as in \cite{dieker_reflected_2009,
foddy_analysis_1984,
Foschini,franceschi_asymptotic_2016,
franceschi_explicit_2017}) which takes 
$\mu_1 <0$
and $\mu_2 <0$.



\begin{figure}[htb]
    \centering
    \subfigure[Recurrent cases]{
\includegraphics[scale=0.5]{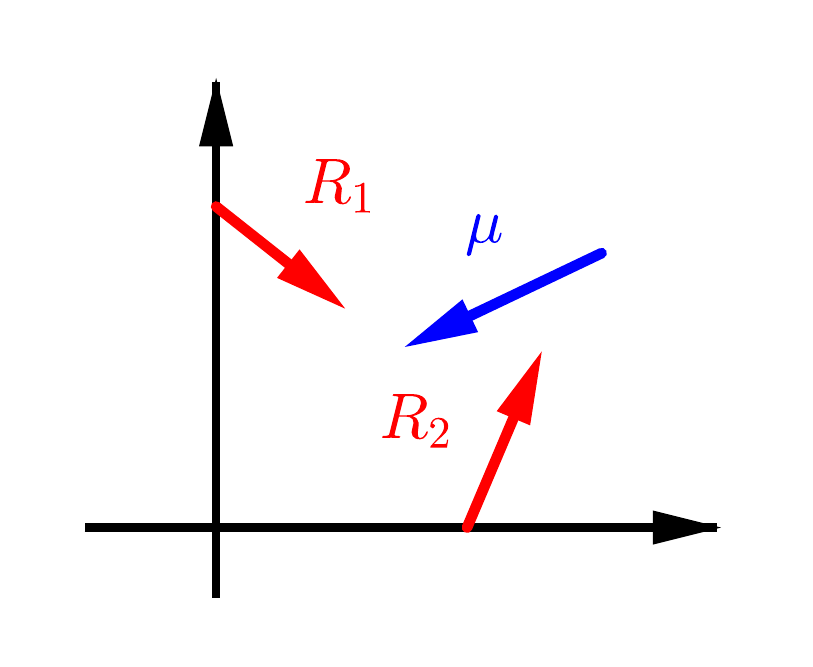}
\includegraphics[scale=0.5]{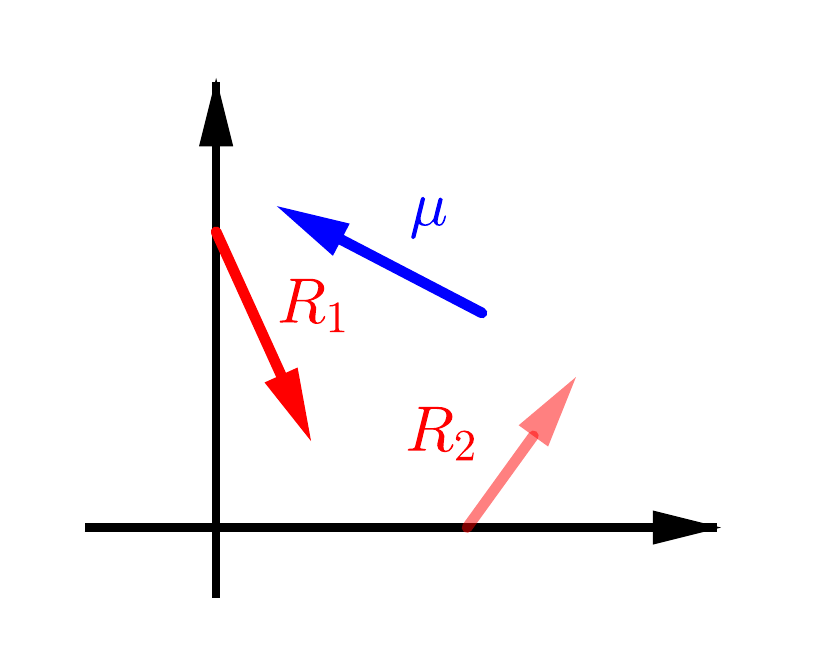}
\includegraphics[scale=0.5]{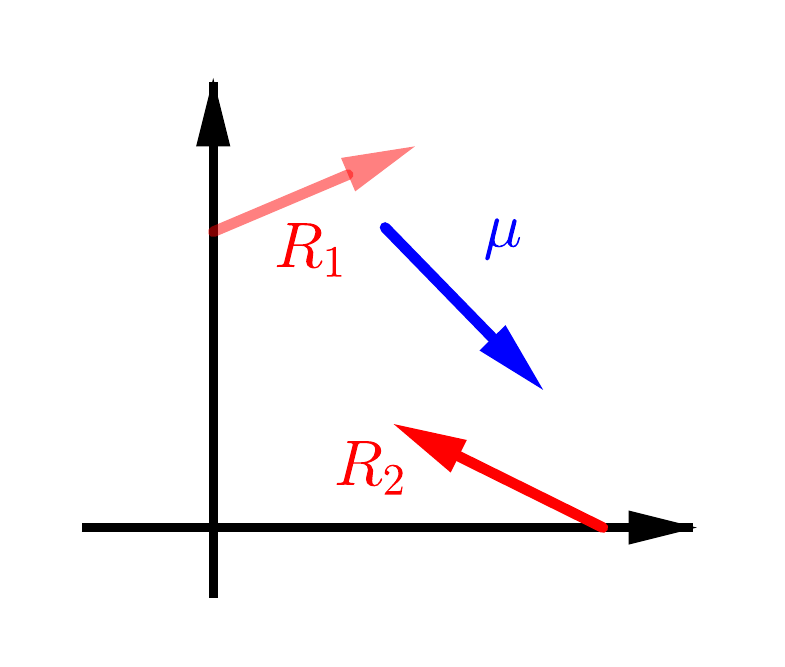}    
    }
    \subfigure[Transient cases]{\includegraphics[scale=0.5]{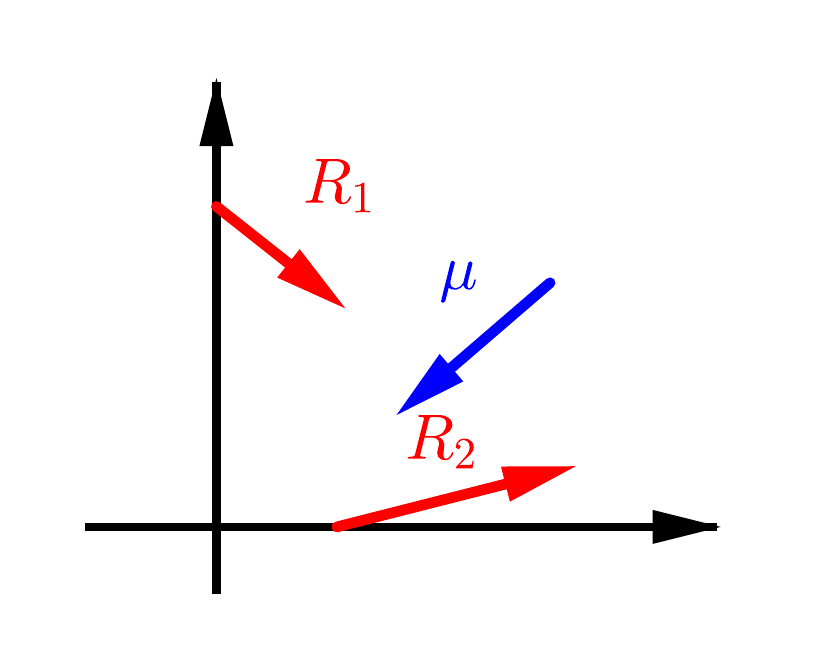} 
\includegraphics[scale=0.5]{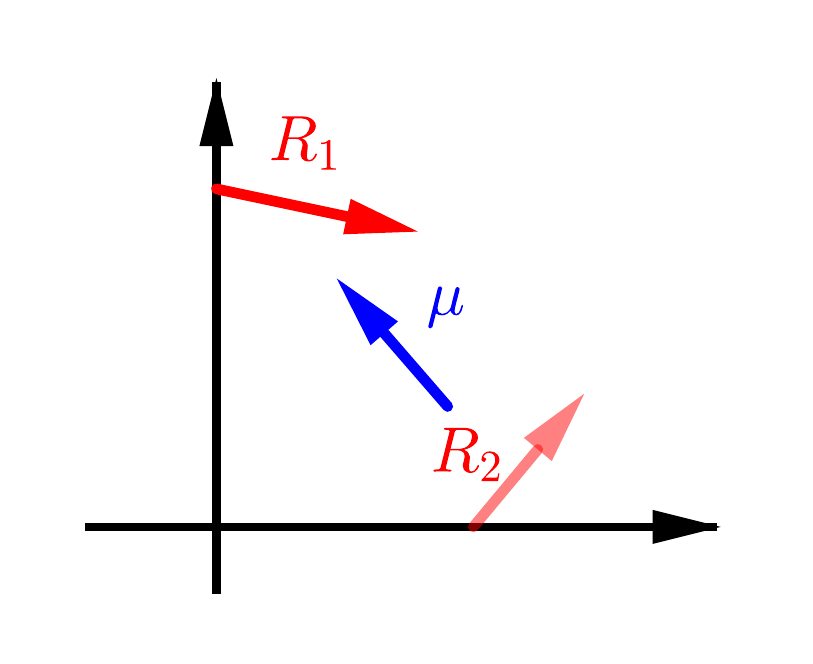} 
\includegraphics[scale=0.5]{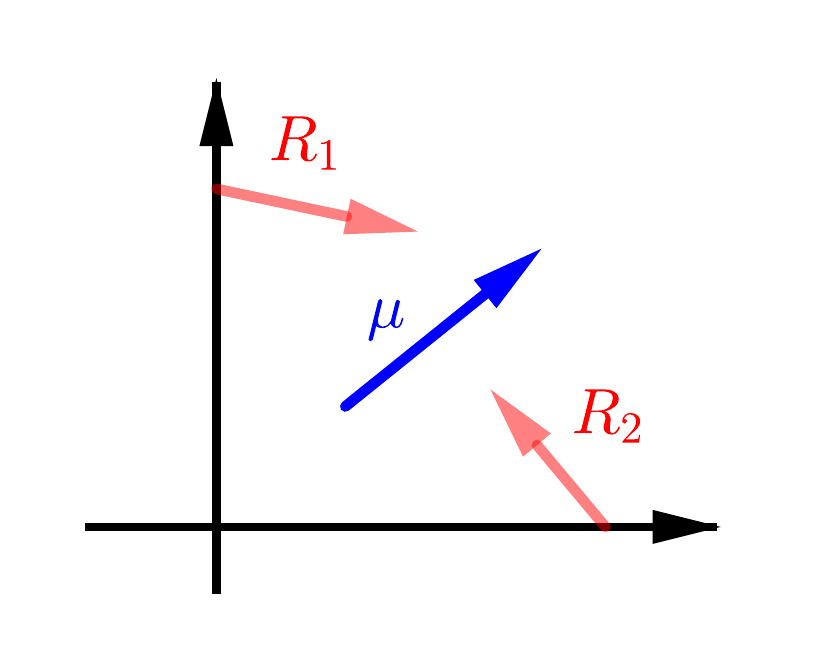}}
    \caption{Recurrence and transience conditions according to the parameters}
    \label{fig:recurrenttransient}
\end{figure}

\section{A new functional equation}
\label{sec:functionalequation}

%
%
 
\subsection{Functional equation} 
\label{subsec:functional_equation}
 
We determine a kernel functional equation which is the starting point of our analytic study.
This key formula connects the Laplace transforms of \remst{the }Green's function inside and on the boundaries of the quarter plane.
Let us define \rems{the \textit{kernel}} $\gamma$, $\gamma_1$ and $\gamma_2$ the two 
variables polynomials such that for $\theta=(\theta_1,\theta_2)$ we have
\begin{align}
\label{intro:def:gamma}
  \begin{cases}
     \gamma (\theta)=\frac{1}{2}  \theta \cdot \Sigma \theta +  \theta \cdot \mu  = \frac{1}{2}(\sigma_{11}\theta_1^2 + 2\sigma_{12}\theta_1\theta_2 + \sigma_{22} \theta_2^2) + \mu_1\theta_1+\mu_2\theta_2, \\
     \gamma_1 (\theta)=  R^1 \cdot \theta =\theta_1 + r_{21} \theta_2,  \\
     \gamma_2 (\theta)= R^2 \cdot\theta =r_{12} \theta_1 + \theta_2,
  \end{cases}
\end{align}
where $\cdot$ is the scalar product.
The equations $\gamma=0$, $\gamma_1=0$ and $\gamma_2=0$ respectively define in $\mathbb{R}^2$ an ellipse and two straight  lines. 
Let $\theta^*$ (resp. $\theta^{**}$) be the point in $\mathbb{R}^2 \setminus (0,0)$ such that $\gamma(\theta^*)=0$ and $\gamma_1(\theta^*)=0$ (resp. $\gamma_2(\theta^{**})=0$). The point $\theta^*$ (resp. $\theta^{**}$) is the intersection point between the ellipse $\gamma=0$ and the straight line $\gamma_1=0$ (resp. $\gamma_2=0$), see Figure \ref{fig:ellipse}. 
\begin{rem}
\label{rem:driftnormal}
Notice that the drift $\mu$ is an \rems{outer normal} vector to the ellipse in $(0,0)$. 
Then the ellipse 
$\{\theta\in \mathbb{R}^2 : \gamma(\theta)=0 \}
\subset \{\theta\in \mathbb{C}^2 : 
\Re\theta \cdot \mu <0   \}
$.
\end{rem}
\begin{figure}[hbtp]
\centering
\includegraphics[scale=0.6]{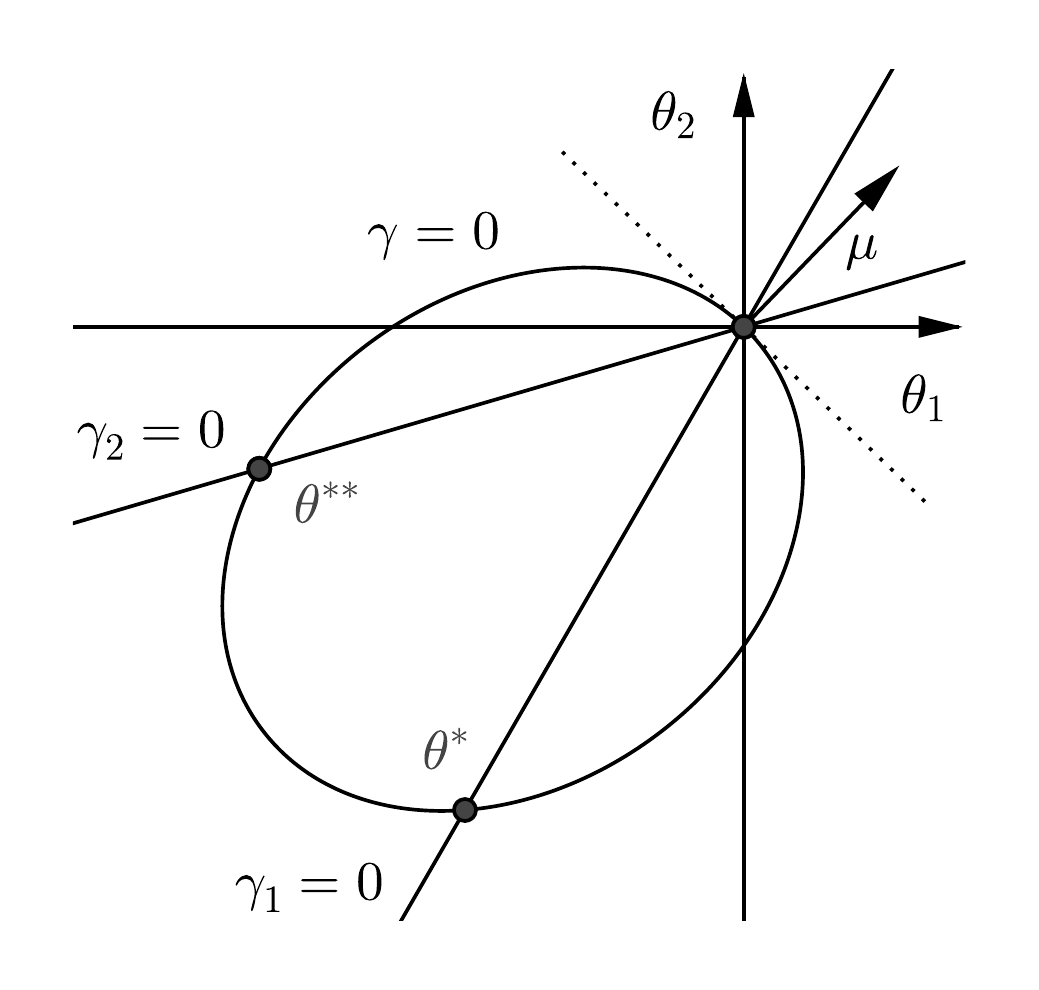}
\caption{Ellipse $\gamma=0$, straight lines $\gamma_1=0$ and $\gamma_2=0$ and intersection points $\theta^*$ and $\theta^{**}$}
\label{fig:ellipse}
\end{figure}

\begin{prop}[Functional equation]
\label{propeqfoncgreen}
\rems{Assume that $\mu_1>0$ and $\mu_2>0$.}
\remst{In the transient case, that is when ${{Z(t)}}$ tends to infinity a.s.,}\rems{Denoting by} $x$ the starting point \rems{of the transient process $Z$}, the following formula holds
\begin{equation}
\label{eq:functional_eq_green}
- \gamma (\theta) \psi (\theta) = \gamma_1 (\theta) \psi_1 (\theta_2) + \gamma_2 (\theta) \psi_2 (\theta_1) + e^{\theta \cdot x}
\end{equation}
for all $\theta=(\theta_{1},\theta_2 )\in \mathbb{C}^2$ such that \rems{$\Re \theta \cdot \mu <0 $ and such that} the integrals $\psi(\theta)$, $\psi_1(\theta_2)$ and $\psi_2(\theta_1)$ are finite.
Furthermore\remst{ when $\mu_1>0$ and $\mu_2>0$}:
\begin{itemize}
\item 
$\psi_1(\theta_2) $ is finite on $\{\theta_2\in \mathbb{C} : \Re\theta_2\leqslant \theta_2^{**} 
\}
$,
\item
$\psi_2(\theta_1)$ is finite on $\{\theta_1\in \mathbb{C} : \Re\theta_1 \leqslant \theta_1^* 
\}
$,
\item $\psi(\theta)$ is finite on $\{\theta\in \mathbb{C}^2 : 
\Re\theta_1 < \theta_1^*\wedge 0 \text{ and } \Re\theta_2 <\theta_2^{**}\wedge 0
  \}
\rems{  \subset \{\theta\in \mathbb{C}^2 : 
\Re\theta \cdot \mu <0   \}}
$.
\end{itemize}
\end{prop}

\begin{proof}
The proof of this functional equation is a consequence of Ito's formula.
For $f\in \mathcal{C}^2 (\mathbb{R}_+^2)$ 
we have
\begin{equation*} 
 f({\rems{Z(t)}})-f(Z_0)=\int_0^t \nabla f({\rems{Z(s)}})\rems{\cdot}\mathrm{d} W_s + \int_0^t \mathcal{L}f({\rems{Z(s)}}) \, \mathrm{d} s + \sum_{i=1}^2 \int_0^t  R_i \cdot \nabla f({\rems{Z(s)}}) \, \mathrm{d} {\rems{L_i(t)}},
\end{equation*}
where $\mathcal{L}$ is the generator defined in \eqref{eq:def_L_generateur}.
Choosing $f(z)=e^{\theta \cdot z}$ for $z\in\mathbb{R}_+^2$ and taking the expectation of the last equality we obtain :
\begin{equation}
\mathbb{E}_{x} [ e^{\theta \cdot {\rems{Z(t)}}} ] - e^{\theta \cdot x}
=
0 
+ 
\gamma (\theta)   \mathbb{E}_{x} \left[ \int_0^t  e^{\theta\cdot {\rems{Z(s)}}} \, \mathrm{d} s \right]
+ 
\sum_{i=1}^2 \gamma_i (\theta)  \mathbb{E}_{x} \left[  \int_0^t e^{\theta\cdot {\rems{Z(s)}}} \, \mathrm{d} {\rems{L_i(t)}} \right].
\label{eq:democonvergence}
\end{equation}
Indeed 
$\int_0^t \nabla f({\rems{Z(s)}})\rems{\cdot}\mathrm{d} W_s $ is a martingale 
and then its expectation is zero.
Now let $t$ tend to infinity. 
Due to \eqref{eq:drift_positive_condition} we have $\theta \cdot {\rems{Z(t)}} / t \underset{t\to\infty}{\longrightarrow}\theta \cdot \mu$ 
Choosing $\theta$ such that\remst{ $\psi(\theta)$ is finite} \rems{$\Re\theta \cdot \mu <0$ then} implies that $\Re\theta \cdot {\rems{Z(t)}} \to -\infty$. 
We deduce that $\mathbb{E}_{x} [e^{\theta \cdot {\rems{Z(t)}}} ] \underset{t\to\infty}{\longrightarrow} 0 $. The expectations of the following formula being finite by 
hypothesis, we obtain 
$$
0 - e^{\theta \cdot x}
=
\gamma (\theta)   \mathbb{E}_{x} \left[ \int_0^{\infty}e^{\theta \cdot {\rems{Z(s)}}} \, \mathrm{d} s \right]
+ 
\sum_{i=1}^2 \gamma_i (\theta)  \mathbb{E}_{x} \left[ \int_0^{\infty}e^{\theta \cdot {\rems{Z(s)}}} \,  \mathrm{d}  {\rems{L_i(t)}} \right] 
$$
which is the desired equation~\eqref{eq:functional_eq_green}.

Let us now assume that $\theta=\theta^*$ in equality \eqref{eq:democonvergence}, we obtain
$$
\mathbb{E}_{x} [ e^{\theta^* \cdot {\rems{Z(t)}}} ] - e^{\theta^* \cdot x}
=
\gamma_2 (\theta^*)  \mathbb{E}_{x} \left[  \int_0^t e^{\theta^*\cdot {\rems{Z(s)}}} \, \mathrm{d} {\rems{L_2(t)}} \right].
$$
Let $t$ \rems{tend} to infinity. Thanks to Remark \ref{rem:driftnormal} we have $\theta^* \cdot {\rems{Z(t)}}\to -\infty$ and we obtain
$$
\psi_2(\theta_1^*)=
\mathbb{E}_{x} \left[  \int_0^\infty e^{\theta^*\cdot {\rems{Z(s)}}} \, \mathrm{d} {\rems{L_2(t)}} \right]
=\frac{- e^{\theta^* \cdot x}}{
\gamma_2 (\theta^*)}  <\infty.
$$
It implies that $\psi_2(\theta_1)$ is finite for all $\Re\theta_1 \leqslant \theta_1^*$. On the same way we obtain that $\psi_1(\theta_2)$ is finite for all $\theta_2 \leqslant \theta_2^{**}$. 

Now assume that $\theta$ satisfies $\Re\theta_1< \theta_1^* \wedge 0$, $\Re\theta_2< \theta_2^{**}\wedge 0$ 
and let us deduce that the Laplace transform
 $\psi(\theta_1,\theta_2)$ is finite. 
Thanks to \eqref{eq:drift_positive_condition} we have $ \Re \theta \cdot \mu <0$ and then $\mathbb{E}_{x} [e^{\theta \cdot {\rems{Z(t)}}} ] \underset{t\to\infty}{\longrightarrow} 0 $. 
Let us consider two cases:
\begin{itemize}
\item if $\gamma (\theta_1^*\wedge 0, \theta_2^{**}\wedge 0) \neq 0$, taking $\theta=(\theta_1^*\wedge 0, \theta_2^{**}\wedge 0)$ and letting $t$ \rems{tend} to infinity in \eqref{eq:democonvergence}, we obtain that $\psi(\theta_1^*\wedge 0, \theta_2^{**}\wedge 0)$ is finite. Then $\psi(\theta_1,\theta_2)$ is finite for all $(\theta_1,\theta_2)$ such that $\Re\theta_1 \leqslant \theta_1^*\wedge 0$ and $\Re\theta_2\leqslant \theta_2^{**}\wedge 0$.
\item if $\gamma (\theta_1^*\wedge 0, \theta_2^{**}\wedge 0) = 0$ it is possible to find $\epsilon>0$ as small as we want such that
$\gamma (\theta_1^*\wedge 0-\epsilon, \theta_2^{**}\wedge 0-\epsilon) \neq 0$. \rems{In} the same way that in the previous case we deduce that $\psi(\theta_1,\theta_2)$ is finite for all $(\theta_1,\theta_2)$ such that $\Re\theta_1< \theta_1^*\wedge 0$ and $\Re\theta_2< \theta_2^{**}\wedge 0$.
\end{itemize} 
\end{proof}

\subsection{Kernel}
\label{subsec:noyau}

The kernel $\gamma$ defined in \eqref{intro:def:gamma} can be written as
\begin{equation*}
\label{eq:alternative_expression_kernel_BM}
     \gamma(\theta_1, \theta_2)
     =a(\theta_1)\theta_2^2+b(\theta_1)\theta_2+c(\theta_1),
\end{equation*}
where $a,b,c$ 
are polynomials in $\theta_1$ such that 
\begin{equation*}
a(\theta_1)= \frac{1}{2}\sigma_{22},
\qquad
b(\theta_1)=\sigma_{12}\theta_1+\mu_2,
\qquad
c(\theta_1)=\frac{1}{2}\sigma_{11}\theta_1^2+\mu_1\theta_1.
\end{equation*}
Let $d (\theta_1) = b^2 (\theta_1) -4a(\theta_1)c(\theta_1)$ be the discriminant. It has two real zeros $\theta_1^\pm$ of opposite sign
which are equal to
\begin{equation}
\theta_1^\pm= \frac{(\mu_2\sigma_{12}-\mu_1\sigma_{22}) \pm \sqrt{(\mu_2\sigma_{12}-\mu_1\sigma_{22})^2 +\mu_2^2 \det{\Sigma}}}{\det\Sigma}.
\label{eq:definition_theta_pm}
\end{equation}
\remst{Cancelling the kernel }We define $\Theta_2(\theta_1)$ a bivalued algebraic function which has two branch points $\theta_1^\pm$ by $\gamma(\theta_1,\Theta_2(\theta_1))= 0$. 
We define the two branches $ \Theta_2^\pm$ on the cut plane $\mathbb{C}\setminus ((-\infty,\theta_1^-)\cup(\theta_1^+,\infty))$ by
$\Theta_2^\pm(\theta_1)
     =\frac{-b(\theta_1)\pm\sqrt{d(\theta_1)}}{2a(\theta_1)} $, that is
\begin{align}
     \Theta_2^\pm(\theta_1)
&=
\dfrac{-(\sigma_{12}\theta_1+\mu_2)\pm\sqrt{\theta_1^2(\sigma_{12}^2-\sigma_{11}\sigma_{22})+2\theta_1(\mu_2\sigma_{12}-\mu_1\sigma_{22})+\mu_2^2}}{\sigma_{22}}    .
\label{eq:definition_Theta_pm}
\end{align}
On $(-\infty,\theta_1^-)\cup(\theta_1^+,\infty)$ the discriminant $d$ is negative and the branches $\Theta_2^\pm$ take conjugate complex values on this set. It will imply that the curve $\R$ defined in equation \eqref{eq:curve_definition1} is \rems{symmetric with respect to} the \rems{horizontal axis}.
On the same way we define $\theta_2^\pm$ and $\Theta_1^\pm$, it yields
$$
\theta_2^\pm = \frac{(\mu_1\sigma_{12}-\mu_2\sigma_{11}) \pm \sqrt{(\mu_1\sigma_{12}-\mu_2\sigma_{11})^2 +\mu_1^2 \det{\Sigma}}}{\det\Sigma}
$$
and
\begin{equation}
\label{eq:definition_Theta_pm1}
\Theta_1^\pm(\theta_2)
=\dfrac{-(\sigma_{12}\theta_2+\mu_1)\pm\sqrt{\theta_2^2(\sigma_{12}^2-\sigma_{11}\sigma_{22})+2\theta_2(\mu_1\sigma_{12}-\mu_2\sigma_{11})+\mu_1^2}}{\sigma_{11}}.
\end{equation}
\rems{The previous formulas can also be found in \cite[(7) and (8)]{franceschi_explicit_2017}.}

\subsection{Holomorphic continuation}
\label{subsec:continuation}

The boundary value problem satisfied by $\psi_1(\theta_2)$ in Section~\ref{sec:BVP} lies on a curve outside of the convergence domain established in Proposition~\ref{propeqfoncgreen} that is $\{ \theta_2\in\mathbb{C} : \Re \theta_2 \leqslant \theta_2^{**} \}$. That is why we extend holomorphically the Laplace transform $\psi_1$.
We assume that the transient condition \eqref{eq:transient_condition} is satisfied.
\begin{lem}[Holomorphic continuation]
\label{lem:continuation_BM}
The Laplace transform $\psi_1$ may be holomorphically extended to the open 
set 
\begin{equation}
\label{eq:domain_continuation}
     \{\theta_2\in\mathbb{C}\setminus (\theta_2^+,\infty) : \Re
    \, \theta_2
    < \theta_2^{**} 
      \text{ or } \Re\, \Theta_1^-(\theta_2) < \theta_1^*\}.
\end{equation}
\end{lem}

\begin{proof} \rems{This proof is similar
to the one of Lemma 3 of \cite{franceschi_explicit_2017}.}
The Laplace transform $\psi_1$ is initially defined on $\{\theta_2\in\mathbb{C} : \Re
    \, \theta_2
    < \theta_2^{**} 
     \}$, see Proposition \ref{propeqfoncgreen}.
By evaluating the functional equation \eqref{eq:functional_eq_green} at $(\Theta_1^-(\theta_2),\theta_2)$\remst{ we cancel the kernel and} we have
\begin{equation}
\label{eq:prolongement}
     \psi_1(\theta_2)=-\frac{\gamma_2(\Theta_1^-(\theta_2),\theta_2)\psi_2(\Theta_1^-(\theta_2))
+
\exp (\Theta_1^-(\theta_2) x_1 + \theta_2 x_2)}{\gamma_1(\Theta_1^-(\theta_2),\theta_2)}   
\end{equation}
for $\theta_2$ in the open and non-empty set 
$\{\theta_2\in\mathbb{C} : \Re
    \, \theta_2
    < \theta_2^{**} 
      \text{ and } \Re\, \Theta_1^-(\theta_2) < \theta_1^*\}$. The formula \eqref{eq:prolongement} then allows to continue meromorphically $\psi_1$ on $\{\theta_2\in\mathbb{C} :  \Re\, \Theta_1^-(\theta_2) < \theta_1^*\}$. The potential poles may come from the zeros of $\gamma_1(\Theta_1^-(\theta_2),\theta_2)$. The points $\theta^*$ and $(0,0)$ are the only points \rems{at which $\gamma_1$ is $0$}. We notice that $\Theta_1^-(0)\neq 0$ as $\mu_1>0$. Then the only possible value in that domain \rems{at which the denominator of \eqref{eq:prolongement} takes the value $0$} is $\theta_2^*$  when $\theta^*=(\Theta_1^-(\theta_2^*),\theta_2^*)$. In that case $\theta_2^*<\theta_2^{**}$ and thanks to Proposition \ref{propeqfoncgreen} we deduce that $\psi_1(\theta_2^{*})$ is finite (which means that the numerator of \eqref{eq:prolongement} is zero). We conclude that $\psi_1$ is holomorphic in the domain \eqref{eq:domain_continuation}.
\end{proof}
This continuation is similar to what is done for the Laplace transform of the invariant measure in \cite{franceschi_explicit_2017,franceschi_tuttes_2016,
franceschi_asymptotic_2016}. In fact it would be possible to introduce the Riemann surface $\mathcal{S}=\{(\theta_1,\theta_2)\in\mathbb{C}^2:\gamma(\theta_1,\theta_2)=0\}$ which is a sphere and to continue meromorphically the Laplace transforms to the whole surface and even on its universal covering.

%
%


\section{A boundary value problem}
\label{sec:BVP}

The goal of this section is to establish and to solve the 
non-homogeneous Carleman boundary value problem with shift satisfied by $\psi_1 (\theta_2)$, the Laplace transform of Green's function on the \rems{vertical axis}. Here the shift is the complex conjugation. We will refer to the reference books on boundary value problems \cite{litvinchuk_solvability_2000,Mu-72,gakhov_boundary_66}
and one will see Appendix \ref{appendix:BVP} 
for a brief survey 
of this theory.
In this section we will assume that transience condition \eqref{eq:transient_condition} is satisfied. 

\subsection{Boundary and domain}

This section is mostly technical. Before to state the BVP in Section \ref{subsec:carlemanBVP} we need to introduce the boundary $\R$ and the domain $\G$ where the BVP will be satisfied.

\subsubsection*{An hyperbola}
The curve $\R$ is a branch of hyperbola already introduced in \cite{baccelli_analysis_1987,franceschi_explicit_2017,franceschi_tuttes_2016}. We define $\R$ as
\begin{equation}
\label{eq:curve_definition1}
     \R =\{\theta_2\in\mathbb C: \gamma(\theta_1,\theta_2)=0 \text{ et } \theta_1\in(-\infty,\theta_1^-)\}=\Theta_2^\pm ((-\infty,\theta_1^-))
\end{equation}
and $\G$ as the open domain of $\mathbb{C}$ bounded by $\R$ on the right, 
see Figure \ref{BVPtheta1}.
As we noticed in Section \ref{subsec:noyau} the curve $\R$ is \rems{symmetric} with respect to the horizontal axis, see Figure~\ref{BVPtheta1}. See 
\cite{franceschi_explicit_2017,franceschi_tuttes_2016} or \cite[Lemma 9]{baccelli_analysis_1987} for more details and a study of this hyperbola. In particular the equation of the hyperbola is given by
\begin{equation}
\label{eq:hyperbole}
     \sigma_{22}(\sigma_{12}^2-\sigma_{11}\sigma_{22})x^2+\sigma_{12}^2\sigma_{22}y^2-2\sigma_{22}(\sigma_{11}\mu_2-\sigma_{12}\mu_1)x=\mu_2(\sigma_{11}\mu_2-2\sigma_{12}\mu_1).
\end{equation}
\rems{In} Figure \ref{fig:domainprolongement} one can see the shape of $\R$ according to the sign of the covariance $\sigma_{12}$. The part of $\R$ with negative imaginary part is denoted by $\R^-$.

\begin{figure}[hbtp]
\centering
\includegraphics[scale=0.25]{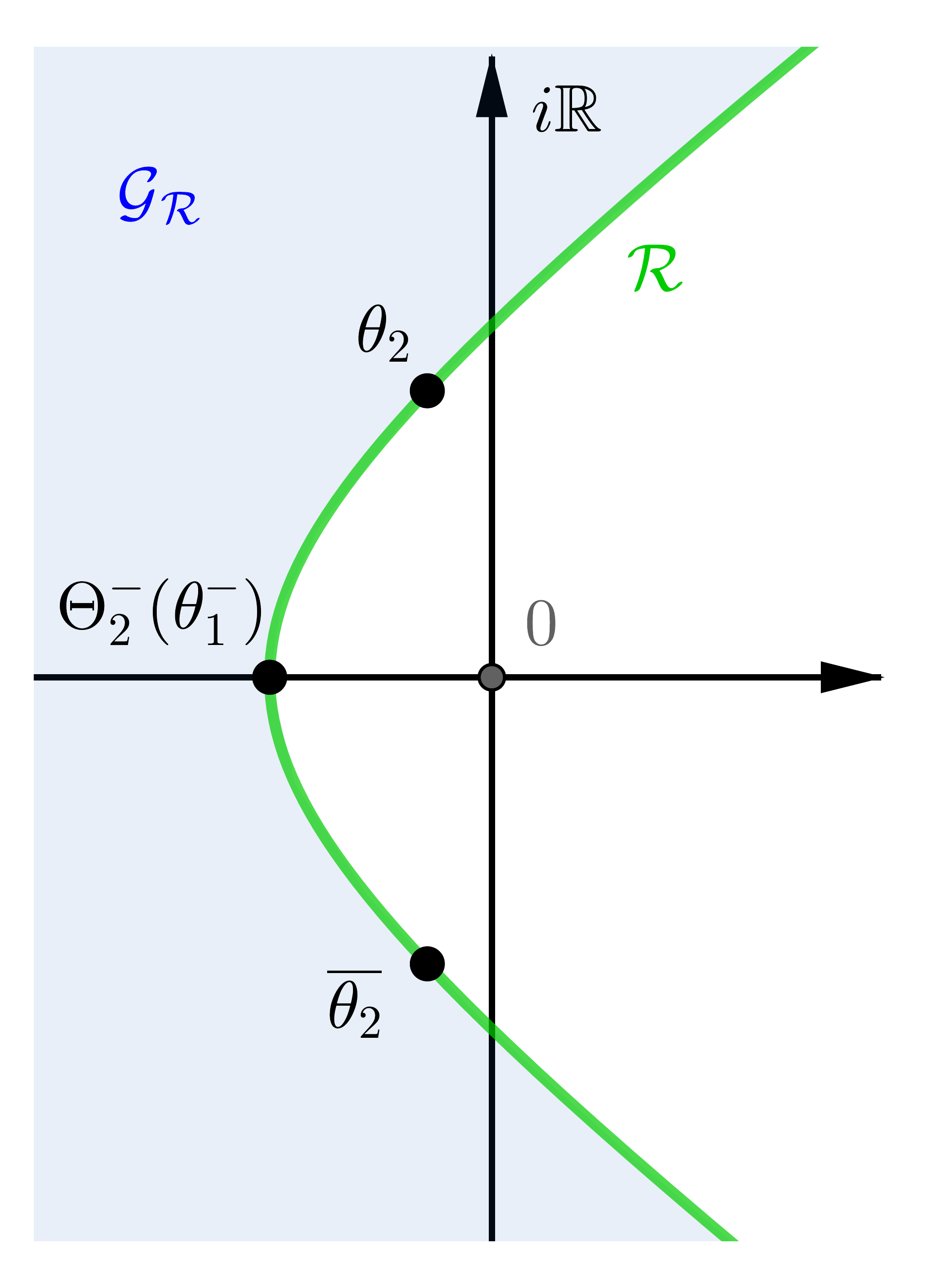}
\caption{Curve $\R$ defined in~\eqref{eq:curve_definition1} in green and domain $\G$ in blue}
\label{BVPtheta1}
\end{figure}

\subsubsection*{Continuation on the domain}
Together with Lemma \ref{lem:continuation_BM} the following lemma implies that $\psi_1$ may be holomorphically extended to a domain containing $\overline{\G}$.
\begin{lem}
\label{lem:domainincluded}
The set $\overline{\G}$ is strictly included in the domain 
$$
\{\theta_2\in\mathbb{C}\setminus (\theta_2^+,\infty) : \Re
    \, \theta_2
    < \theta_2^{**} 
      \text{ or } \Re\, \Theta_1^-(\theta_2) < \theta_1^*\}
      $$
      defined in \eqref{eq:domain_continuation}.
\end{lem}
\begin{proof}
This proof is similar 
to the one of Lemma 5 of \cite{franceschi_explicit_2017}.
First we notice that the set $\overline{\G} \cap \{\theta_2\in\mathbb{C} : \Re\, \theta_2 < \theta_2^{**} \}$ is included in the domain defined in \eqref{eq:domain_continuation}.
Then it remains to prove that the set 
$$S:= \overline{\G} \cap \{\theta_2\in\mathbb{C}  : \Re\, \theta_2 \geqslant \theta_2^{**}\}$$ is a subset of the domain \eqref{eq:domain_continuation}. 
More precisely, we show that $S$ is included in 
\begin{equation*}
    T:= \{\theta_2\in\mathbb{C} \setminus (\theta_2^+,\infty) : \Re\, \Theta_1^-(\theta_2) < \theta_1^*\}.
\end{equation*}
                                                                                                                                                                                                                                                                                                                                                                                                                                                                                                                                                                                                                                                                                                                                                                                                                                                                                      First of all, notice that the set $S$ is bounded by (a part of) the hyperbola $\R$ and (a part of) the straight line $\theta_2^{**}+i\mathbb{R}$. We denote 
$\theta_2^{**}\pm i t_1$ the two intersection points of these two curves when they exist, see Figure~\ref{fig:domainprolongement}.
The definition of $\R$ implies that $\R \subset T$. 
Indeed the image of $\R$ by $\Theta_1^-$ is included in $(-\infty,\theta_1^-)$ and $\theta_1^-\leqslant\theta_1^*$.
Furthermore (the part of) $\theta_2^{**}+i\mathbb{R}$ that bounds $S$ also belongs to $T$ because for $t\in\mathbb{R}_+$ and using the fact that $\det \Sigma >0$ Equation \eqref{eq:definition_Theta_pm1} yields after some calculations
$$\begin{cases*}
    \Re \Theta_1^-(\theta_2^{**}\pm it)\leqslant
 \Re \Theta_1^-(\theta_2^{**})=\theta_1^{**}     
       <\theta_1^{*} , 
       & when $\theta_2^{**}\leqslant\Theta_2(\theta_1^-)$;
       \\
       \Re \Theta_1^-(\theta_2^{**}\pm i(t_1+t))\leqslant
 \Re \Theta_1^-(\theta_2^{**}\pm i t_1)    
       <\theta_1^{*}  ,
       & when $\theta_2^{**}>\Theta_2(\theta_1^-)$.
\end{cases*}$$
The inequality $\theta_1^{**}     
       <\theta_1^{*}$ \rems{follows from} the assumption that $\theta_2^{**}\leqslant\Theta_2(\theta_1^-)$ 
and the inequality $ \Re \Theta_1^-(\theta_2^{**}\pm i t_1)    
       <\theta_1^{*}$ \rems{follows from} the fact that $\theta_2^{**}\pm i t_1\in\R\subset T$. Let us denote \rems{$\beta=\arccos { \left( -\frac{\sigma_{12} }{\sqrt{\sigma_{11}\sigma_{22}}} \right)}$}.
To conclude we consider two cases:
\begin{itemize}
     \item $\sigma_{12}<0$ or equivalently $0<\beta<\frac{\pi}{2}$: the set $S$ is either empty or bounded, see the left picture on Figure~\ref{fig:domainprolongement}. Applying the maximum principle to the function $\Re \Theta_1^-$ 
     show that the image of every point of $S$ by $\Re \Theta_1^-$ is smaller than $\theta_1^*$ and then that $S$ is included in $T$.
     \item$\sigma_{12}\geqslant 0$ or equivalently $\frac{\pi}{2}\leqslant \beta<\pi$:
henceforth the set $S$ is unbounded as we can see on the right picture of Figure \ref{fig:domainprolongement}. It is \rems{no longer} possible to apply directly the maximum principle. However, to conclude we show that the image by $\Re \Theta_1^-$ of a point $re^{it}\in T$ near to infinity is smaller than $\theta_1^*$.
%
    The asymptotic directions of $\theta_1^*+i\mathbb{R}$ are $\pm \frac{\pi}{2}$ and \eqref{eq:hyperbole} 
    implies that those of $\R$ are $\pm(\pi-\beta)$. Then as in the proof of Lemma 5 of \cite{franceschi_explicit_2017} we prove with \eqref{eq:definition_Theta_pm1} that 
    for $ t\in  (\pi-\beta, \frac{\pi}{2})$ we have
\begin{equation*}
\Theta_1^-(re^{\pm it})\underset{r\to\infty}{\sim}
r\sqrt{\frac{\sigma_{22}}{\sigma_{11}}} e^{\pm i(t+\beta)}.
\end{equation*}
   For $ t\in  (\pi-\beta, \frac{\pi}{2})$ this implies that 
    $\Re\Theta_1^-(re^{\pm it})\underset{r\to\infty}{\longrightarrow}-\infty $ and we obtain that $\Re \Theta_1^-(re^{\pm it})<\theta_1^*$ for $r$ large enough. As in the case $\sigma_{12}<0$ we finish the proof with the maximum principle.\qedhere
\end{itemize}
\end{proof}

\begin{figure}[hbtp]
\centering
\includegraphics[scale=0.5]{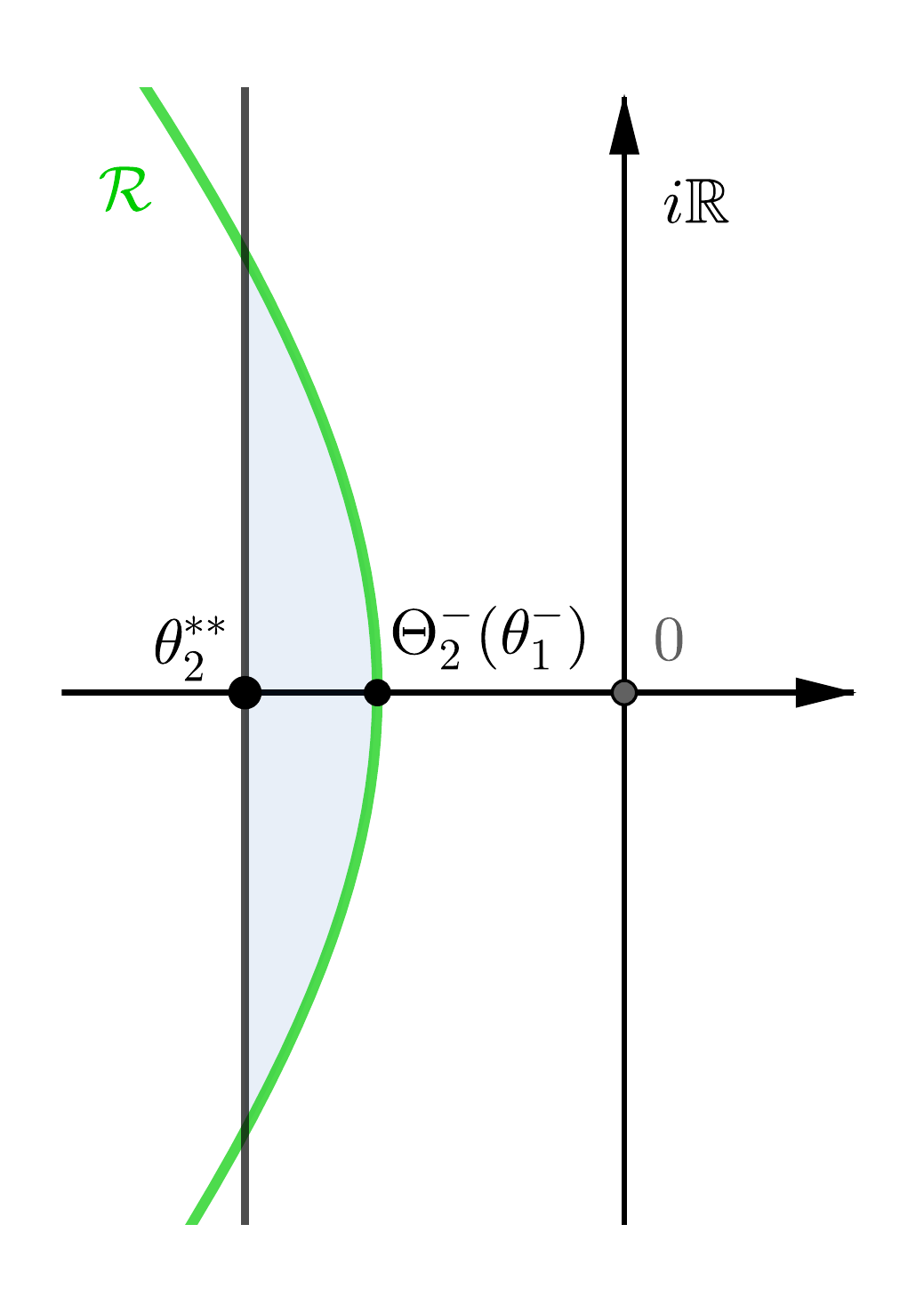}\qquad
\includegraphics[scale=0.5]{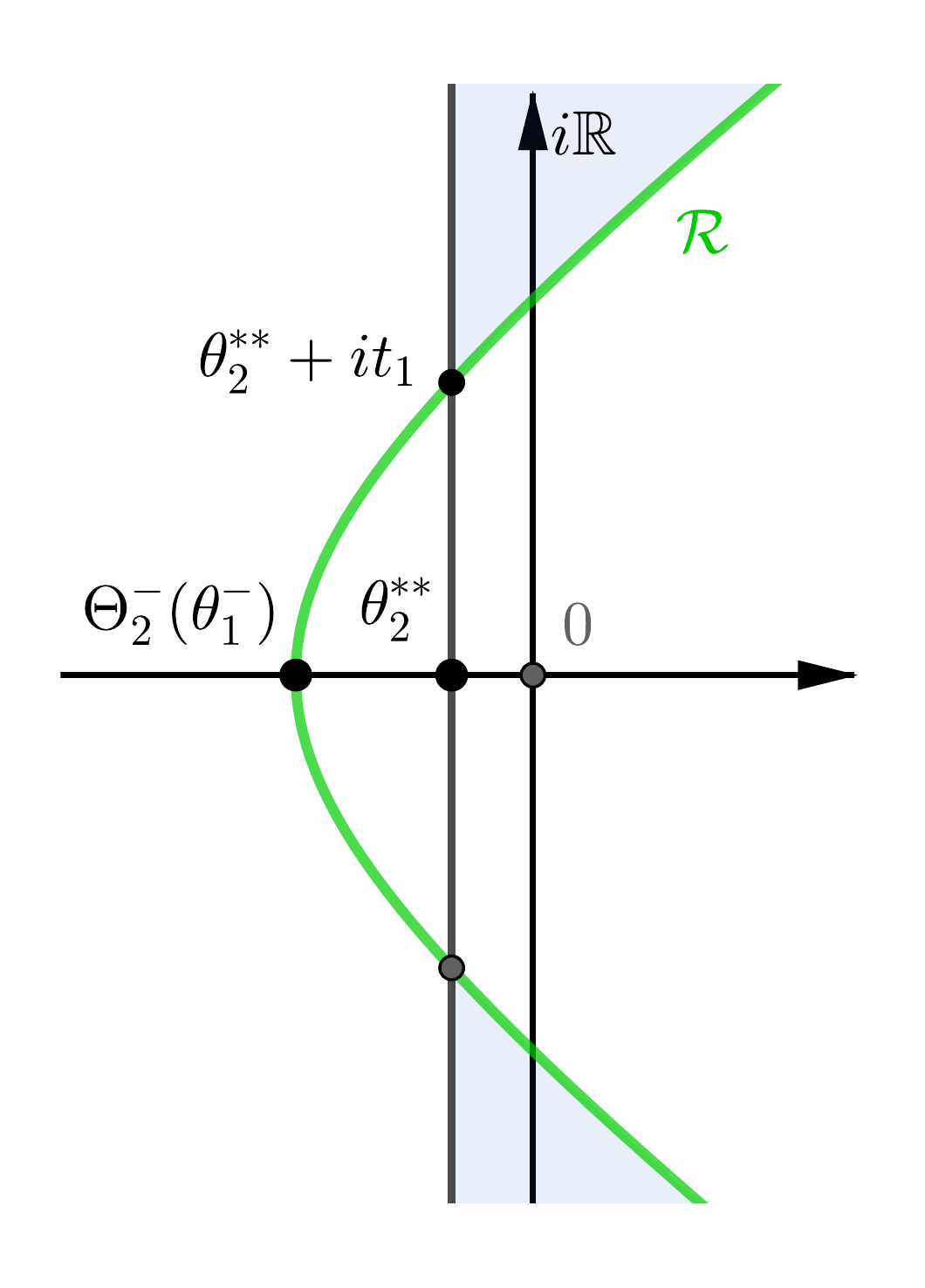}
\caption{On the left $\sigma_{12}<0$, and on the right $\sigma_{12}\geqslant 0$. The blue domain is the set $S$. 
}
\label{fig:domainprolongement}
\end{figure}

\subsection{Carleman boundary value problem}
\label{subsec:carlemanBVP}
We establish a \textit{boundary value problem} (BVP) with shift (here it is the complex conjugation) on the hyperbola $\R$. 
Let us define the functions $G$ and $g$ such that
\begin{align}
\label{eq:def:G}
     G(\theta_2)&:=\frac{\gamma_1}{\gamma_2}(\Theta_1^-(\theta_2),\theta_2)\frac{\gamma_2}{\gamma_1}(\Theta_1^-(\theta_2),\overline{\theta_2}),
\\
\label{eq:def:g}
g(\theta_2)&:=
\frac{\gamma_2}{\gamma_1}(\Theta_1^-(\theta_2),\overline{\theta_2})
\left(
\frac{e^{(\Theta_1^-(\theta_2),\theta_2) \cdot x}}{\gamma_2(\Theta_1^-(\theta_2),\theta_2)}
-
\frac{e^{(\Theta_1^-(\theta_2),\overline{\theta_2})\cdot x}}{\gamma_2(\Theta_1^-(\theta_2),\overline{\theta_2})}
\right)
.
\end{align}

\begin{lem}[BVP for $\psi_1$]
\label{lem:BVP}
The Laplace transform $\psi_1$ satisfies the following boundary value problem:
\begin{enumerate}[label={\rm (\roman{*})},ref={\rm (\roman{*})}]
     \item \label{BVPi} $\psi_1$ is analytic on $\G$, continuous on its closure $\overline{\G}$ and \rems{tends} to $0$ at infinity;   
     \item \label{BVPii} $\psi_1$ 
     satisfies the boundary condition
\begin{equation}
\label{eq:boundary_condition_general1}
     \psi_1(\overline{\theta_2})=G(\theta_2)\psi_1({\theta_2}) + g(\theta_2), \qquad \forall \theta_2\in \R.
\end{equation}
\end{enumerate}
\end{lem}
This BVP is said to be non-homogeneous because of the function 
$g$ coming from the term $e^{\theta \cdot x}$ in the functional equation.

\begin{proof}
The analytic and continuous properties of item \ref{BVPi} \rems{follow} from Lemma \ref{lem:continuation_BM} and Lemma~\ref{lem:domainincluded}. The behaviour at infinity \rems{follows} from the integral formula \eqref{eq:defpsi} which defines the Laplace transform $\psi_1$ and from the continuation formula \eqref{eq:prolongement}. 
We now show item \ref{BVPii}. For $\theta_1\in (-\infty,\theta_1^-)$ let us evaluate the functional equation \eqref{eq:functional_eq_green} at the points $(\theta_1,\Theta_2^\pm(\theta_1)$. It yields the two equations
$$
0 = \gamma_1 (\theta_1,\Theta_2^\pm(\theta_1)) \psi_1 (\Theta_2^\pm(\theta_1)) + \gamma_2 (\theta_1,\Theta_2^\pm(\theta_1)) \psi_2 (\theta_1) + e^{(\theta_1,\Theta_2^\pm(\theta_1)) \cdot x}
.
$$
Eliminating $\psi_2 (\theta_1)$ from the two equations gives
\begin{align*}
 \psi_1 (\Theta_2^+(\theta_1))
 =
 &\frac{\gamma_1}{\gamma_2}
 (\theta_1,\Theta_2^-(\theta_1))
  \frac{\gamma_2}{\gamma_1}
  (\theta_1,\Theta_2^+(\theta_1))
  \psi_1 (\Theta_2^-(\theta_1))
  \\ &+
  \frac{\gamma_2}{\gamma_1}
  (\theta_1,\Theta_2^+(\theta_1))
\frac{e^{(\theta_1,\Theta_2^-(\theta_1)) \cdot x}}{\gamma_2(\theta_1,\Theta_2^-(\theta_1))}
-
\frac{e^{(\theta_1,\Theta_2^+(\theta_1) \cdot x}}{\gamma_1(\theta_1,\Theta_2^+(\theta_1))}.
\end{align*}
Choosing $\theta_1\in (-\infty,\theta_1^-)$, the quantities $\Theta_2^+(\theta_1)$ and $\Theta_2^-(\theta_1)$ go through the whole curve $\R$ (defined in \eqref{eq:curve_definition1}) and are complex conjugate, see Section \ref{subsec:noyau}. Noticing in that case that $\Theta_1^-(\Theta_2^-(\theta_1))=\theta_1$, we obtain equation \eqref{eq:boundary_condition_general1}. 
\end{proof}

\subsection{Conformal \rems{glueing} function}
\label{subsec:collage}

To solve the BVP of Lemma \ref{lem:BVP} we need a function $w$
which satisfies the following conditions:
\begin{enumerate}[label={\rm (\roman{*})},ref={\rm (\roman{*})}]
     \item\label{item:conformal_10} $w$ is holomorphic on $\G$, continuous on $\overline{\G}$ and \rems{tends} to infinity at infinity,
     \item\label{item:conformal_20} $w$ is one to one from $\G$ to $\mathbb{C}\setminus (-\infty ,-1]$, 
     \item\label{item:conformal_30} $w(\theta_2)=w(\overline{\theta_2})$ for all $\theta_2\in\mathcal{R}$.
\end{enumerate}
Such a function $w$ is called a \textit{conformal \rems{glueing} function} because it glues together the upper and the lower part of the hyperbola $\R$. 
Let us define $w$ in terms of generalized Chebyshev polynomial 
\begin{align}
\notag
T_a(x)  &:=\cos (a\arccos (x))=\frac{1}{2} \rems{\big[} \big(x+\sqrt{x^2-1}\big)^a+\big(x-\sqrt{x^2-1}\big)^a \rems{\big]},
\\
\beta &:=\arccos \rems{{
    \left( -\frac{\sigma_{12}}{\sqrt{\sigma_{11}\sigma_{22}}}
\right)    
    }},
     \label{eq:def_beta}
\\
\label{eq:expression_CGF_BM1}
     {w} (\theta_2) &:=T_{\frac{\pi}{\beta}}\bigg(-\frac{2\theta_2-(\theta_2^++\theta_2^-)}{\theta_2^+-\theta_2^-}\bigg)
     \text{for all } \theta_2\in\mathbb{C}\setminus [\theta_2^+,\infty).
\end{align}

The function $w$ is a conformal \rems{glueing} function which satisfies \ref{item:conformal_10}, \ref{item:conformal_20}, \ref{item:conformal_30} and $w(\Theta_2^\pm (\theta_1^-))=-1$.
See \cite[Lemma 3.4]{franceschi_tuttes_2016} for the proof of these properties. The following lemma is a direct consequence of these properties.
\begin{lem}[Conformal \rems{glueing} function]
The function $W$ defined by 
$$
W(\theta_2)=\frac{w(\theta_2)+1}{w(\theta_2)}
$$
satisfies the following properties :
\begin{enumerate}
     \item\label{item:conformal_1} $W$ is holomorphic on $\G\setminus \{w^{-1}(0)\}$, continuous on $\overline{\G}\setminus \{w^{-1}(0)\}$ and \rems{tends} to $1$ at infinity,
     \item\label{item:conformal_2} $W$ is one to one from $\G\setminus \{w^{-1}(0)\}$ to $\mathbb{C}\setminus [0,1]$, 
     \item\label{item:conformal_3} $W(\theta_2)=W(\overline{\theta_2})$ for all $\theta_2\in\mathcal{R}$.
\end{enumerate}
\end{lem}
We introduce $W$ to avoid any technical problem at infinity. We have a cut on the segment $[0,1]$ and we will be able to apply the propositions presented in Appendix \ref{appendix:BVP}.
Notice that we have chosen 
arbitrarily the pole of $W$ in $w^{-1}(0)$, but every other point $w^{-1}(x)$ for $x\in\mathbb{C}\setminus (-\infty,-1]$ would have been suitable.


\subsection{Index of the BVP}
\label{subsec:index_BVP}

We denote \remst{Comment to the referee: $H$ is defined in the appendix B, but in fact there is no need to use it here so I remove it}
$$
\Delta=[\text{arg } G ]_{\mathcal{R}^-}  \remstf{=[\text{arg } H ]_0^1}
\quad \text{and} \quad
d=\text{arg } G (\Theta_2^\pm (\theta_1^-))\remstf{=\text{arg} H(0)} \in (-\pi,\pi].
$$
To solve the BVP of Lemma \ref{lem:BVP} we need to compute the index $\chi$ which is defined by
$$
\chi=\left\lfloor \frac{d+\Delta}{2\pi} \right\rfloor .
$$
\begin{lem}[Index]
\label{lem:index}
The index $\chi$ is equal to
$$
\chi=
\begin{cases}
0 & \text{if } \gamma_1(\theta_1^-, \Theta_2^\pm (\theta_1^-))  \gamma_2(\theta_1^-, \Theta_2^\pm (\theta_1^-)) \leqslant 0,
\\ 
1 & \text{if } \gamma_1(\theta_1^-, \Theta_2^\pm (\theta_1^-))  \gamma_2(\theta_1^-, \Theta_2^\pm (\theta_1^-)) > 0.
\end{cases}
$$
\end{lem}
The index is then equal to $0$ or $1$ depending on the position of the two straight lines $\gamma_1=0$ and $\gamma_2=0$  with respect to the red point $(\theta_1^-, \Theta_2^\pm (\theta_1^-)) $. See Figure \ref{fig:chi} which illustrates this lemma. 
\begin{proof}
The proof is similar in each step to the proof of Lemma 14 in \cite{franceschi_explicit_2017} except that in our case $ \gamma_2(\theta_1^-, \Theta_2^\pm (\theta_1^-))$ is not always positive.
\end{proof}

\begin{figure}[hbtp]
\centering
\includegraphics[scale=0.6]{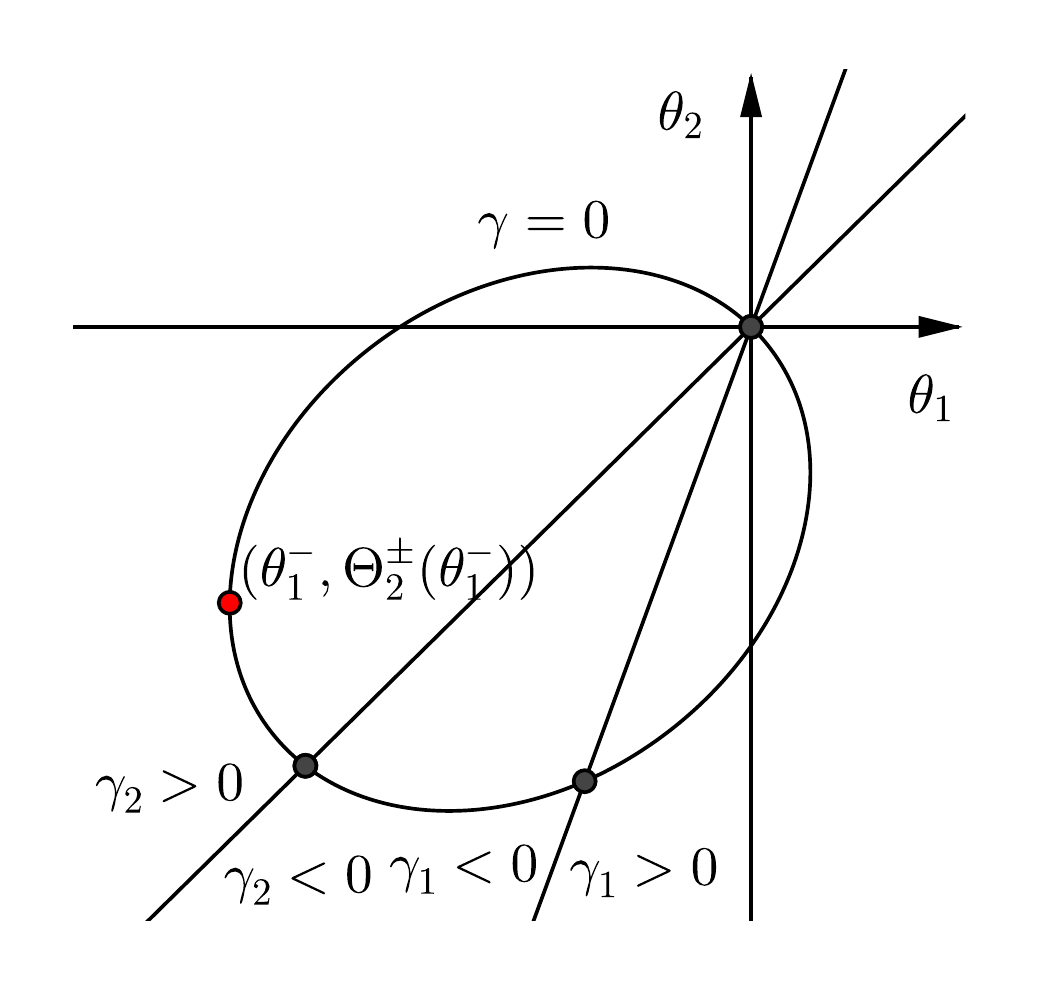}
\includegraphics[scale=0.6]{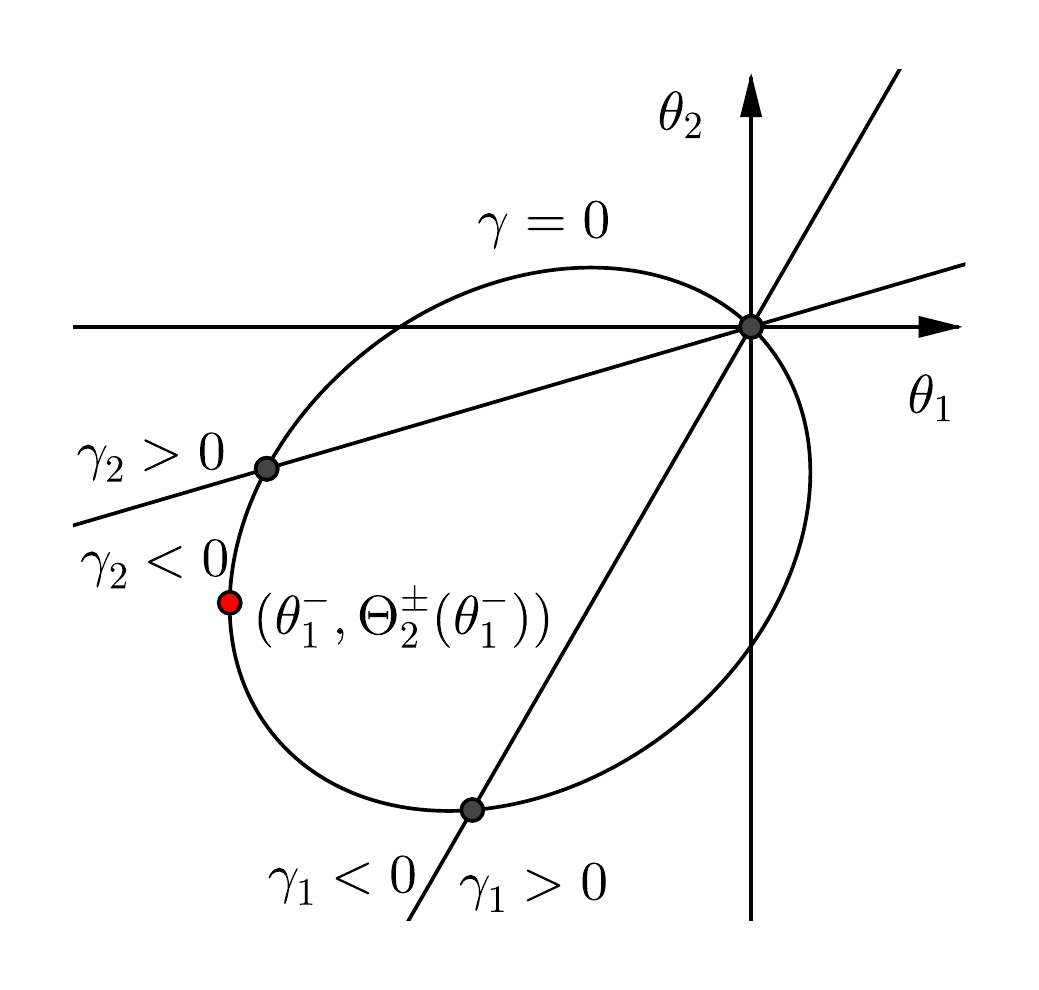}
\caption{On the left $\chi=0$ and on the right $\chi=1$}
\label{fig:chi}
\end{figure}

\subsection{Resolution of the BVP}
\label{subsec:resolution_BVP}

The following theorem, already presented in the introduction as the main result of this paper, holds.
\begin{thm}[Explicit expression of $\psi_1$]
\label{thm:main}
Assume conditions \eqref{eq:transient_condition} and \eqref{eq:drift_positive_condition}. For $\theta_2\in\G$, the Laplace transform $\psi_1$ defined in \eqref{eq:defpsi} is equal to
\begin{equation}
\psi_1(\theta_2)=
\frac{-Y (w(\theta_2))}{2i\pi} \int_{\R^-} \frac{g(t)}{Y^+ (w(t))}
\left(
 \frac{w'(t) }{w(t)-w(\theta_2)}
 + \chi
  \frac{ w'(t) }{w(t)} 
  \right)
 \,  \mathrm{d}t
 \label{eq:main}
\end{equation}
with 
$$
Y (w(\theta_2))=w(\theta_2)^{\chi} 
\exp  \left(
\frac{1}{2i\pi} \int_{\R^-} \log(G(s))\left(\frac{ w'(s)}{w(s)-w(\theta_2)}- \frac{ w'(s)}{w(s)} \right) \, \mathrm{d}s
\right),
$$
and where 
\begin{itemize}
\item $G$ is defined in \eqref{eq:def:G} and $g$ in \eqref{eq:def:g},
\item $w$ is the conformal \rems{glueing} function defined in \eqref{eq:expression_CGF_BM1},
\item$\R^-$ is the part of the hyperbola $\R$ defined in \eqref{eq:curve_definition1} with negative imaginary part,
\item $\chi=0$ or $1$ is determined by Lemma \ref{lem:index},
\item $Y^+$ is the limit value on $\R^-$ of $Y$ 
(and may be expressed thanks to Sokhotski-Plemelj formulas stated in Proposition \ref{prop:sokhotski} of Appendix \ref{appendix:BVP}).
\end{itemize}
\end{thm}
\begin{proof}
We define the function $\Psi$ by 
$$
\Psi(z)=\psi_1(W^{-1}(z)), \quad \text{for } z\in\mathbb{C}\setminus[0,1].
$$
Then $\Psi$ satisfies the Riemann BVP of Proposition \ref{C:prop:pbmriemannassocie} in Appendix \ref{appendix:BVP}. The resolution of this BVP leads to Proposition \ref{C:prop_solution_carleman} which gives a formula for the Laplace transform $\psi_1 = \Psi \circ W$. We then have
$$
\psi_1(\theta_2)=
\begin{cases}
X(W(\theta_2))\phi(W(\theta_2)) & \text{for } \chi=-1,
\\
 X(W(\theta_2))(\phi(W(\theta_2))+C) & \text{for } \chi=0,
 \end{cases}
$$
where $C$ is a constant, 
$\chi$ is determined in Lemma \ref{lem:index} and the functions $X$ and $\phi$ are defined by 
$$
X(W(\theta_2)):= (W(\theta_2)-1)^{-\chi} 
\exp  \left(
\frac{1}{2i\pi} \int_{\R^-} \log(G(t))\frac{ W'(t)}{W(t)-W(\theta_2)} \, \mathrm{d}t
\right), \quad  \theta_2\in\G,
$$
and
$$
\phi (W(\theta_2)):= \frac{-1}{2i\pi} \int_{\R^-} \frac{g(t)}{X^+ (W(t))} \frac{W'(t) }{W(t)-W(\theta_2)} \, \mathrm{d}t, \quad  \theta_2\in\G.
$$
When $\chi=0$ the constant is determined evaluating $\psi_1$ at $-\infty$. We have $\psi_1(-\infty)=0$, $W(-\infty)=0$ 
and we obtain $C=-\phi(1)=\frac{1}{2i\pi} \int_{\R^-} \frac{g(t)}{X^+ (W(t))} \frac{W'(t) }{W(t)-1} \, \mathrm{d}t$.
To end the proof we just have to notice that $$W(\theta_2)-1=\frac{1}{w(\theta_2)},
\quad 
\frac{W'(t)}{W(t)-W(\theta_2)}=
\frac{w'(t)}{w(t)-w(\theta_2)}
-
\frac{w'(t)}{w(t)}
\quad \text{and}
\quad
\frac{W'(t)}{W(t)-1}=
-
\frac{w'(t)}{w(t)}.
$$
\end{proof}

\subsection{Decoupling functions}
\label{subsec:decoupling_functions}

Due to the function $G\neq1$ in \eqref{eq:boundary_condition_general1}, the boundary value problem is complex. When it is possible to reduce the BVP to the case where $G=1$, it is then possible to solve it directly thanks to Sokhotski-Plemelj formulas, see Remark \ref{C:rem:sechol} in Appendix \ref{appendix:BVP}. 

In some specific cases it is possible to find a rational function $F$ satisfying the decoupling condition
\begin{equation}
\label{eq:def:decouplage_F}
G(\theta_2) = \frac{F(\theta_2)}{F(\overline{\theta_2})}, \quad \forall\theta_2\in\mathcal{R},
\end{equation}
where $G$ is defined in \eqref{eq:def:G}. Such a function $F$ is called a decoupling function. In \cite{BoElFrHaRa_algebraic_2018} the authors show that such a function exists if and only if the following condition holds
\begin{equation}
\label{eq:condition_decouplage}
\epsilon+\delta\in \beta\mathbb{Z}+\pi\mathbb{Z},
\end{equation}
where $\beta$ is defined in \eqref{eq:def_beta} and
$\epsilon,\delta\in(0,\pi)$ are defined by
\begin{equation*}
\label{eq:expression_delta_epsilon}
     \tan\epsilon=
     \frac{\sin\beta}{r_{21}\sqrt{\frac{\sigma_{11}}{\sigma_{22}}}+\cos\beta}
     \quad
\text{and}
\quad
     \tan\delta=
     \frac{\sin\beta}{r_{12}\sqrt{\frac{\sigma_{22}}{\sigma_{11}}}+\cos\beta}.
\end{equation*}
In this case it is possible to solve in an easier way the boundary value problem. The boundary condition \eqref{eq:boundary_condition_general1} may be rewritten as
$$
(F\psi_1)(\overline{\theta_2})=(F\psi_1)({\theta_2}) + F(\overline{\theta_2})g(\theta_2), \qquad \forall \theta_2\in \R.
$$
Using again the conformal \rems{glueing} function $w$,
we transform the BVP into a Riemann BVP, see Appendix \ref{appendix:BVP}.
Such an approach leads to an alternative formula for $\psi_1$ which is simpler.
Indeed, thanks to Remark \ref{C:rem:sechol}, 
in the cases where the rational fraction $F$ \rems{tends} to $0$ at infinity, 
we obtain 
$$
\psi_1({\theta_2})
=
\frac{1}{2i\pi}
\frac{1}{F(\theta_2)}
\int_{\mathcal{R}^-}
\frac{F(\overline{t})g(t)}{w(t)-w(\theta_2)} \, \mathrm{d}t, \quad \theta_2\in\G.
$$



\appendix

\section{Potential theory}
\label{appendix:potentialtheory}

There have not been many studies to determine explicit expressions for Green's functions of diffusions. In order to make the article self-contained and give context,
in this appendix we illustrate in an informal way the links between partial differential equations and Green's functions of Markov processes in potential theory. 

\subsection{Dirichlet boundary condition and killed process}
Let $\Omega$ be an open, bounded, smooth subset of $\mathbb{R}^d$ and $X$ an homogeneous diffusion of generator $\mathcal{L}$ starting from $x$ and killed at the boundary $\partial\Omega$. Assume that $X$ admits a \rems{transition density} $p_t(x,y)$ and \rems{denote by} $g(x,y)$ the Green's function defined by 
$$g(x,y)=
\int_0^{\infty} p_t(x,y) \, \mathrm{d} t.$$
The forward Kolmogorov equation (or Fokker-Planck equation) with boundary and initial condition says that 
$$
\begin{cases}
\mathcal{L}^*_y p_t(x,y) = \partial_t p_t(x,y), 
\\
p_t(x,\cdot)=0 \text{ on } \partial\Omega,
\\
p_0 (x,\cdot) = \delta_x.
\end{cases} 
$$
Integrating this equation in time we can see that Green's function is a fundamental solution of the dual operator $L^*$ and satisfies
$$
\begin{cases}
\mathcal{L}^*_y g (x,\cdot)= -\delta_x & \text{in } \Omega,
\\
g (x,\cdot) = 0 & \text{on } \partial \Omega.
\end{cases} 
$$
Now, let $f$ be a continuous function on $\overline{\Omega}$ and $\phi$ a continuous function on $\partial \Omega$. If we assume that the equation
$$
\begin{cases}
\mathcal{L} u = -f & \text{in } \Omega,
\\
u = \phi & \text{on } \partial\Omega,
\end{cases}
$$
admits a unique solution, it is possible to express it in terms of \remst{the }Green's functions. We have
$$
u(x)= \mathbb{E}_x \left[\int_0^\tau  f({\rems{X(t)}}) \,  \mathrm{d}t\right]  + \mathbb{E}_x \left[ \phi(X_\tau) \right] 
=\int_\Omega f(y) g(x,y)\, \mathrm{d}y
+\int_{\partial \Omega} \phi(y) \partial_{n_y}g(x,y) \, \rems{\mathrm{d}y},
$$
where $\tau$ is the first exit time of $\Omega$.
(Note that $\partial_{n_y}g$, the \textit{inner} 
normal derivative on the boundary of Green's function, may be interpreted as the density of the distribution of the exit place.)
Thanks to Green's functions 
it is then possible to solve an interior Poisson's type equation with Dirichlet boundary conditions which 
specify the value of $u$ on the boundary and the value of $\mathcal{L}u$ inside $\Omega$. 

\subsection{Neumann boundary condition and reflected process}
Henceforth, let us replace the interior Dirichlet problem by an exterior Neumann boundary problem which specifies 
the value of the normal derivative of $u$ on the boundary and the value of $\mathcal{L}u$ outside $\Omega$ in $\Omega^c =\mathbb{R}^d \setminus \overline{\Omega}$ :
$$
\begin{cases}
\mathcal{L} u = -f & \text{in } \Omega^c,
\\
\partial_{n} u  = \phi & \text{on } \partial\Omega.
\end{cases}
$$
While the Dirichlet equation was linked to  
some killed process on the boundary, the Neumann equation is linked to 
a reflected process.
From now, let us denote $X$ the reflected process on $\partial\Omega$ of generator $L$ inside $\Omega^c$. Let us recall that $\Omega^c$ is unbounded, we assume that the process is transient and we note $g$ its Green's function. This time again, $g$ is a fundamental solution of $\mathcal{L}^*$ (with a more complex boundary condition of Robin type linking $u$ and $\partial_n u$). There are some necessary compatibility conditions linking $f$ and $\phi$ in order for a solution to exist, for example if $\mathcal{L}=\Delta$ the interior Neumann boundary problem can have a solution only if $\int_\Omega f =-\int_{\partial\Omega} \phi$. The solution vanishing at infinity of the Neumann problem, if it exists, is equal to 
$$
u(x)=
\mathbb{E}_x \left[\int_0^\infty f({\rems{X(t)}}) \, \mathrm{d}t\right]  + \mathbb{E}_x \left[\int_0^\infty \phi({\rems{X(t)}}) \, \mathrm{d} {\rems{L(t)}}\right] 
=\int_\Omega f(y) g(x,y)\, \mathrm{d}y
+\int_{\partial \Omega} \phi(y) h(x,y) \, \mathrm{d}s(y).
$$
We have noted $L$ the local time that the process spends on the boundary $\partial\Omega$ and $h$ the density of the boundary Green's measure $H$ which is equal to
$$
H(x,A)=\mathbb{E}_x \left[ \int_0^\infty \mathbf{1}_A ({\rems{X(t)}}) \,  \mathrm{d} {\rems{L(t)}} \right]
\quad \text{for } A\subset \partial \Omega
,
$$
and represents the average local time that the process spends on the set $A$ of the boundary. In fact $h$ and the restriction of $g$ to $\partial\Omega$ are intimately related, for example if $\mathcal{L}=\Delta$ then $h=g_{\mid_{\partial \Omega}}$. These formulas present, in an informal way, how to solve a Neumann boundary equation thanks to Green's functions. 
The Appendix \ref{appendix:dim1} illustrates this 
by giving an explicit example in one dimension in \eqref{eq:PDEdim1}. 

Unfortunately, 
finding Green's functions is often a difficult task. 
Notice that in this paper $\Omega^c= \mathbb{R}_+^2$ and $\Omega$ is therefore neither bounded 
nor smooth, \rems{and the reflection is oblique, rather than normal. This makes our task
in this article more complicated.}

\section{Carleman \rems{Boundary Value Problem}}
\label{appendix:BVP}

This appendix is a short presentation
of the boundary value problems (BVP) theory. It introduces methods and techniques 
used for the resolution of BVP. The results presented here can be found in the reference books of \citet{litvinchuk_solvability_2000}, \citet{Mu-72} and \citet{gakhov_boundary_66}.

\subsection{Sokhotski-Plemelj formulae}

Sokhotski-Plemelj formulas are central in the resolution of Riemann boundary value problems. Let $\mathcal{L}$ a contour (open or closed) smooth and oriented and $f\in \mathbb{H}_\mu (\mathcal{L})$ the set of $\mu$-Hölder continuous functions on $\mathcal{L}$ for $0<\mu \leqslant 1$. A function is sectionally holomorphic if it is holomorphic on the whole complex plane \rems{except} $\mathcal{L}$ and admits right and left limits on $\mathcal{L}$ (except on its potential \rems{ends}).
\begin{figure}[hbtp]
\centering
\includegraphics[scale=1]{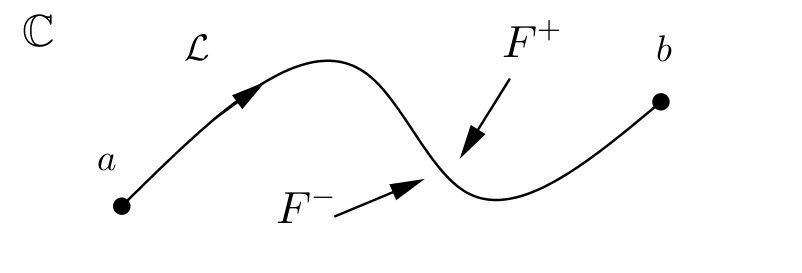}
\caption{An oriented smooth open contour $\mathcal{L}$ of \rems{ends} $a$ and $b$; right limit $F^-$ and left limit $F^+$ of $F$ on $\mathcal{L}$}
\end{figure}

\begin{prop}[Sokhotski-Plemelj formulae]
The function
$$
F (z) := \frac{1}{2i\pi}\int_{\mathcal{L}} \frac{f(t)}{t-z} \, \mathrm{d}t, \quad z\notin \mathcal{L}
$$
is sectionally holomorphic.
The functions $F^+$ and $F^-$ on $\mathcal{L}$ taking the limit values of $F$ respectively on the left and on the right satisfy for $t\in \mathcal{L}$ the formulas 
$$
F^+(t)= \frac{1}{2} f(t) +\frac{1}{2i\pi} \int_{\mathcal{L}} \frac{f(s)}{s-t} \, \mathrm{d}s
\quad \text{and} \quad
F^-(t)= -\frac{1}{2} f(t) +\frac{1}{2i\pi} \int_{\mathcal{L}} \frac{f(s)}{s-t} \, \mathrm{d}s.
$$
Theses formulas are equivalent to the equations
$$
F^+ (t)-F^- (t) =f(t) \quad\text{and}\quad F^+(t)+F^-(t)=\frac{1}{i\pi}
\int_{\mathcal{L}} \frac{f(s)}{s-t} \, \mathrm{d}s.
$$
\label{prop:sokhotski}
\end{prop}
These integrals are understood in the sense of the principal value,  see \cite[Chap. 1,
Sect. 12]{gakhov_boundary_66}.
\begin{rem}[Sectionally holomorphic functions for a given discontinuity]
\label{C:rem:sechol}
Liouville's theorem shows that the function $F$ defined above is the 
unique sectionally holomorphic function $\rems{\Phi}$ satisfying the equation $$\Phi^+ (t)-\Phi^- (t) =f(t), \quad \forall t\in\mathcal{L}$$ and which vanishes at infinity. The solutions of this equation of finite degree at infinity are the functions such that $$\Phi=F+P$$ where $P$ is a polynomial.
\end{rem}
\begin{rem}[Behavior at the \rems{ends}]
\label{C:rem:etrem}
It is possible to show that if $\mathcal{L}$ is an oriented open contour from \rems{end} $a$ to \rems{end} $b$, then \rems{in the
neighborhood} of an \rems{end} $c$ it exists $F_c(z)$, an holomorphic function \rems{in the
neighborhood} of $c$, such that
\begin{equation}
F(z)=\frac{\epsilon_c}{2i\pi} f(c)\log (z-c) + F_c(z)
\quad\text{where}\quad
\epsilon_c=
\begin{cases}
-1 & \text{if } c=a,
\\
1 & \text{if } c=b.
\end{cases}
\end{equation}
\end{rem}

\subsection{Riemann boundary value problem}
\label{C:sec:pbmrieman}

In a standard way, a boundary value problem is composed of a regularity condition on a domain and a boundary condition on that domain. 

\begin{deff}[Riemann BVP]
We say that $\Phi$ satisfies a Riemann BVP on $\mathcal{L}$ if:
\begin{itemize}
\item $\Phi$ is sectionally holomorphic on $\mathbb{C}\setminus\mathcal{L}$ and admits $\Phi^+$ as left limit and $\Phi^-$ as right limit, $\Phi$ if of finite degree at infinity;
\item $\Phi$ satisfies the boundary condition
$$
\Phi^+(t)=G(t)\Phi^-(t) +g(t),
\quad t\in\mathcal{L}
$$
where $G$ and $g$ are functions defined on $\mathcal{L}$.
\end{itemize}
\label{def:riemannBVP}
\end{deff}
We assume here that $G$ and $g\in \mathbb{H}_\mu (\mathcal{L})$ and that $G$ doesn't cancel on $\mathcal{L}$.
When $g=0$ we talk about a homogeneous Riemann BVP.

\subsubsection{Closed contour}

We assume that the contour $\mathcal{L}$ is closed and we denote $\mathcal{L}^+$ the open bounded set of boundary $\mathcal{L}$, and $\mathcal{L}^-$ 
the complementary of $\mathcal{L}^+ \cup \mathcal{L}$.
\begin{figure}[hbtp]
\centering
\includegraphics[scale=1]{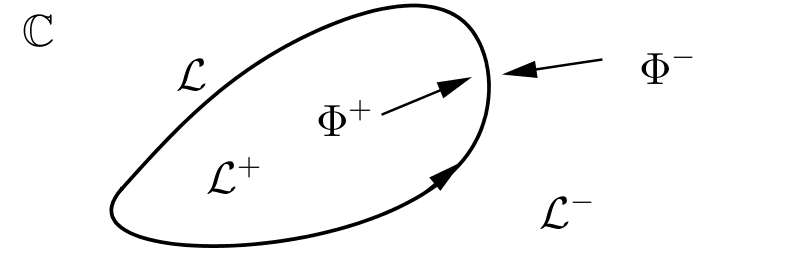}
\caption{Oriented closed smooth contour $\mathcal{L}$, domains $\mathcal{L}^+$ and $\mathcal{L}^-$ and limit values of $\Phi$ on the right and on the left}
\end{figure}

To solve the Riemann BVP we need to introduce the \textit{index}
$$
\chi:=\frac{1}{2i\pi} [\log G]_{\mathcal{L}} = \frac{1}{2\pi} [\arg G]_{\mathcal{L}} 
$$
which quantifies the variation of the argument of $G$ on the contour $\mathcal{L}$ in the positive direction.
Without any loss of generality we assume that $0$ is in $\mathcal{L}^+$. It is then possible to define the single-valued function
$$
\log (t^{-\chi}G(t)), \quad t\in\mathcal{L}
$$
which satisfies the Hölder condition.

\begin{prop}[Solution of homogeneous Riemann BVP on a closed contour]
Let us define
$$
\Gamma (z):= \frac{1}{2i\pi} \int_{\mathcal{L}} \frac{\log(t^{-\chi}G(t))}{t-z} \, \mathrm{d}t, \quad z\notin \mathcal{L}
$$
and
$$
X(z):=
\begin{cases}
\exp \Gamma (z), \quad z\in\mathcal{L}^+,
\\
 z^{-\chi}\exp \Gamma (z), \quad z\in\mathcal{L}^-.
\end{cases}
$$
The function $X$ is the fundamental solution of the homogeneous Riemann BVP of Definition \ref{def:riemannBVP}, i.e. $X$ satisfies the boundary condition $X^+(t)=G(t)X^-(t)$ for $t\in\mathcal{L}$.
The function $X$ is of degree $-\chi$ at infinity. If $\Phi$
is a solution of the homogeneous Riemann BVP, then $\Phi(z)=X(z)P(z)$ where $P$ is a polynomial.
\end{prop}
If we denote $k$ the degree of $P$, the solution $\Phi$ is of degree $k-\chi$ at infinity. The fundamental solution $X$ of degree $-\chi$ is then the non-zero homogeneous solution of smallest degree to infinity.
\begin{proof}
For $t\in \mathcal{L}$, let us denote $\widetilde{\Gamma} (t)= \frac{1}{2i\pi} \int_{\mathcal{L}} \frac{\log(s^{-\chi}G(s))}{s-t} \, \mathrm{d}s$ where the integral is understood in the sense of principal value. Sokhotski-Plemelj formulas applied at $\Gamma$ show that
\begin{equation}
X^+(t)= e^{\widetilde{\Gamma}(t)} \sqrt{t^{-\chi}G(t)}
\quad \text{and} \quad
X^-(t)=t^{-\chi} e^{\widetilde{\Gamma}(t)} \frac{1}{\sqrt{t^{-\chi}G(t)}}
\quad \text{for } t\in\mathcal{L},
\label{eq:defX+}
\end{equation}
and then that $X$ is a solution of the homogeneous problem. If $\Phi$ is a solution of the problem, as $X^\pm (z) \neq 0$ for $z\in\mathcal{L}$ we obtain
$$
\frac{\Phi^+}{X^+}(z)=\frac{\Phi^-}{X^-}(z),
\quad z\in\mathcal{L}.
$$
By analytic continuation the function $\frac{\Phi}{X}$ is then holomorphic in the whole complex plane, is of finite degree at infinity and is then a polynomial according to Liouville's theorem.
\end{proof}

\begin{prop}[Solution of Riemann BVP on a closed contour]
We define
$$
\phi (z):= \frac{1}{2i\pi} \int_{\mathcal{L}} \frac{g(t)}{X^+ (t)(t-z)} \, \mathrm{d}t, \quad z\notin \mathcal{L}.
$$
The solutions of the Riemann BVP of Definition \ref{def:riemannBVP} are the functions such that
$$
\Phi(z)=X(z)\phi(z)+X(z)P_{\chi}(z)
$$
where $P_{\chi}$ is a polynomial of degree $\chi$ for $\chi\geqslant 0$ and $P_{\chi}=0$ for $\chi\leqslant -1$. 
\end{prop}
\begin{rem}[Left limit $X^+$]
We have $X^+(t) =(t-b)^{-\chi} e^{\Gamma^+(t)}$ where 
$\Gamma^+(t)$
is the left limit value of $\Gamma$ on $\mathcal{L}$ given by the Sokhotski-Plemelj formulas of Proposition \ref{prop:sokhotski}, see \eqref{eq:defX+}.
\end{rem}
\begin{rem}[Solubility conditions]
For $\chi< -1$ the solutions are holomorphic at infinity (and then bounded) if and only if the following conditions are satisfied:
\begin{equation}
\int_{\mathcal{L}} \frac{g(t) t^{k-1}}{X^+(t)} \, \mathrm{d}t=0,
\quad k=1, \cdots, -\chi -1.
\label{eq:solubility_condition}
\end{equation}
\end{rem}
\begin{proof}
The fundamental solution $X^\pm$ does not cancel on $\mathcal{L}$ and we have the factorisation $G=\frac{X^+}{X^-}$. If $\Phi$ is a solution of the BVP we have 
$$
\frac{\Phi^+}{X^+}(t)=\frac{\Phi^-}{X^-}(t)+\frac{g}{X^+}(t), \quad t\in\mathcal{L}.
$$
The function $\frac{\Phi}{X}$ being of finite degree at infinity, Remark \ref{C:rem:sechol} gives $\frac{\Phi}{X}=\phi+P$.
\end{proof}

\subsubsection{Open contour}

We assume that the function $\Phi$ we are looking for satisfies the Riemann BVP on an open contour oriented from 
\rems{end} $a$ to \rems{end} $b$ and that $\Phi$ is bounded at the 
\rems{neighborhood} of $a$ and $b$.
More generally, one could look for the solutions admitting singularities integrable at the ends.
We denote $\delta$, $\Delta$, $\rho_a$ and $\rho_b$ such that
$$
G(a)=\rho_a e^{i\delta},
\quad
\Delta = [\arg G]_{\mathcal{L}}
\quad \text{et} \quad
G(b)=\rho_b e^{i(\delta+\Delta)}
$$
choosing $-2\pi < \delta \leqslant 0$ and the corresponding determination of the logarithm $\log G$.
We define the index
$$
\chi:= \left\lfloor \frac{\delta+\Delta}{2\pi} \right\rfloor.
$$
\begin{prop}[Solution of Riemann BVP on an open contour]
Let us define
$$
\Gamma (z):= \frac{1}{2i\pi} \int_{\mathcal{L}} \frac{\log(G(t))}{t-z} \, \mathrm{d}t, \quad z\notin \mathcal{L}.
$$
The function
$$
X(z):= (z-b)^{-\chi} e^{\Gamma(z)}
$$
is a solution of the homogeneous Riemann BVP and is bounded at the \rems{ends}. This solution is of order $-\chi$ at infinity. If $\Phi$ is a solution of the homogeneous problem, it may be written as $\Phi(z)=X(z)P(z)$ where $P$ is a polynomial. We define
$$
\phi (z):= \frac{1}{2i\pi} \int_{\mathcal{L}} \frac{g(t)}{X^+ (t)(t-z)} \, \mathrm{d}t, \quad z\notin \mathcal{L}.
$$
The solutions of the Riemann BVP bounded at the \rems{ends} are the functions
$$
\Phi(z)=X(z)\phi(z)+X(z)P_{\chi}(z)
$$
where $P_{\chi}$ is a polynomial of degree $\chi$ for $\chi\geqslant 0$ and $P_{\chi}=0$ for $\chi\leqslant -1$.
\end{prop}
\begin{proof} 
Due to Remark \ref{C:rem:etrem}, \rems{in the
neighborhood of one end} $c$ we have
$$
e^{\Gamma (z)}=(z-c)^{\lambda_c} e^{\Gamma_c(z)}
$$
for $\Gamma_c$ a holomorphic function \rems{in the
neighborhood} of $c$ and
$$
\lambda_a =-\frac{\delta}{2\pi}+i\frac{\log \rho_a}{2\pi}
\quad \text{and} \quad
\lambda_b =\frac{\delta+\Delta}{2\pi}-i\frac{\log \rho_b}{2\pi}.
$$
Since $\delta \leqslant 0$ the function $e^{\Gamma(z)}$ is bounded at $a$.
Furthermore, we notice that the function $X(z)=(z-b)^{-\chi} e^{\Gamma(z)}$ is bounded at $b$ (and at $a$). The rest of the proof is similar to the closed contour case.
\end{proof}

\subsection{Carleman boundary value problem \rems{with shift}}

\remst{The Carleman BVP is a boundary problem with }
\rems{A \textit{shift} $\alpha(t)$ is a} homeomorphism from the contour $\mathcal{L}$ on itself such that its derivative does not cancel and \rems{which} satisfies Hölder's condition. 
Most of the time  the condition $\alpha(\alpha(t))=t$ is satisfied and we say that $\alpha$ is a \textit{Carleman automorphism} of $\mathcal{L}$. In this paper the shift function is the complex conjugation. 

\begin{deff}[Carleman BVP]
\label{C:def:carleman}
The function $\Phi$ satisfies a Carleman BVP on the closed contour $\mathcal{L}$ (or having its two \rems{ends} at infinity, as in this paper) if:
\begin{itemize}
\item $\Phi$ is holomorphic on the whole domain $\mathcal{L}^+$ bounded by $\mathcal{L}$ and continuous on $\mathcal{L}$; 
\item $\Phi$ satisfies the boundary condition
$$
\Phi(\alpha(t))=G(t)\Phi(t) +g(t),
\quad t\in\mathcal{L},
$$
where $G$ and $g$ are two functions defined on $\mathcal{L}$.
\end{itemize}
\end{deff}
We will assume that $G$ and $g\in \mathbb{H}_\mu (\mathcal{L})$ and that $G$ does not cancel on $\mathcal{L}$.
When $g=0$ the Riemann BVP is said to be homogeneous.

To solve the Carleman BVP 
we introduce a conformal \rems{glueing} function. The following result establishes the existence of such functions. 
\begin{prop}[Conformal \rems{glueing} function]
\label{C:prop:collage}
Let $\alpha$ be a Carleman automorphism of the curve $\mathcal{L}$. It exists $W$, a function
\begin{itemize}
\item holomorphic on $\mathcal{L}^+$ deprived of one point where $W$ has a simple pole;
\item satisfying the \rems{glueing} condition
$$
W(\alpha(t))=W(t),
\quad
t\in\mathcal{L}.
$$
Such a function $W$ establishes a conformal transform (holomorphic bijection) from $\mathcal{L}^+$ to the complex place deprived of a smooth open contour $\mathcal{M}$. This conformal \rems{glueing} function admits two fixed points $A$ and $B$ of image $a$ and $b$ which are the \rems{ends} of $\mathcal{M}$.
\end{itemize}
\end{prop}

If we find such a conformal \rems{glueing} function, we can transform the Carleman BVP into a Riemann BVP.
We orient $\mathcal{M}$ from $a$ to $b$
choosing it such that the orientation of $\mathcal{L}$ be conserved by $W$. We then denote $W^{-1}$ the reciprocal of $W$ and $(W^{-1})^+$ its left limit and $(W^{-1})^-$ its right limit on $\mathcal{M}$. See Figure \ref{C:fig1}.
For $t$ on the arc $\mathcal{L}$ oriented from $B$ to $A$, these functions satisfy
$$
(W^{-1})^+(W(t))=\alpha(t)
\quad \text{and}\quad
(W^{-1})^-(W(t))=t.
$$
\begin{figure}[hbtp]
\centering
\includegraphics[scale=0.9]{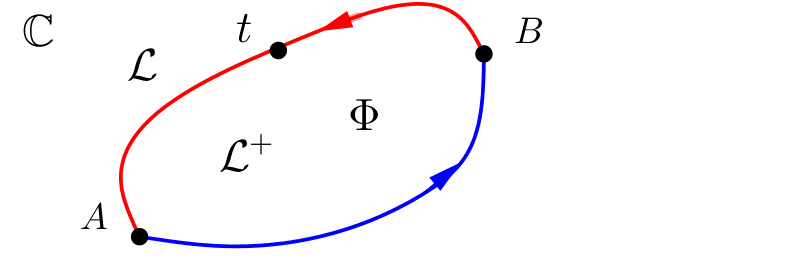}
\includegraphics[scale=0.9]{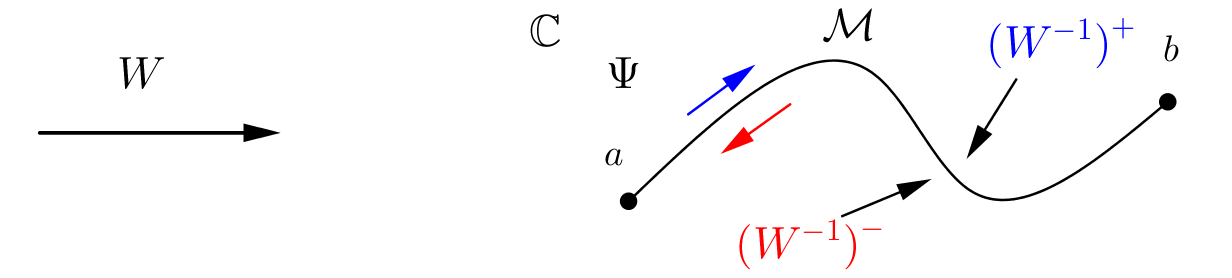}
\caption{Conformal \rems{glueing} function from $\mathcal{L}^+$ to $\mathbb{C}\setminus \mathcal{M}$}
\label{C:fig1}
\end{figure}

Let $\Phi$ be a solution of the Carleman BVP, we define the function $\Psi$ such that
$$
\Psi (W(z)) := \Phi (z),
\quad z\in \mathcal{L}^+.
$$
We then have
$$\Psi(z)= \Phi (W^{-1}(z)),
\quad
z\in\mathbb{C}\setminus \mathcal{M}$$
and the limits on the left and on the right of $\Psi$ on $\mathcal{M}$ are
$$
\Psi^+(t)= \Phi ((W^{-1})^+(t))
\quad \text{and} \quad
\Psi^-(t)= \Phi ((W^{-1})^-(t)),
\quad
t\in \mathcal{M}.
$$
Let
$$
H(t)=G((W^{-1})^-(t))\quad \text{and} \quad
h(t)= g((W^{-1})^-(t)),
\quad
t\in \mathcal{M}.
$$
\begin{prop}
\label{C:prop:pbmriemannassocie}
The function
$\Psi$ satisfies the following Riemann BVP associated to the contour $\mathcal{M}$ and to the functions
$H$ and $h$: 
\begin{itemize}
\item $\Psi$ is sectionally holomorphic on $\mathcal{C}\setminus\mathcal{M}$;
\item $\Psi$ satisfies the boundary condition
$$
\Psi^+(t)=H(t)\Psi^-(t) +h(t),
\quad t\in\mathcal{M}.
$$
\end{itemize}
\end{prop}
\begin{proof}
The proof derives from Definition \ref{C:def:carleman}, from Proposition \ref{C:prop:collage} and from the above notations.
\end{proof}
As $\Phi=\Psi \circ W$, to solve the Carleman BVP of Definition \ref{C:def:carleman}, it is enough to determine the conformal \rems{glueing} function $W$ and to find $\Psi$ thanks to Section \ref{C:sec:pbmrieman} which explains how to solve the Riemann BVP Proposition \ref{C:prop:pbmriemannassocie}. Let us define
$$
X(W(z)):= (W(z)-b)^{-\chi} 
\exp  \left(
\frac{1}{2i\pi} \int_{\mathcal{L}_d} \log(G(t))\frac{ W'(t)}{W(t)-W(z)} \, \mathrm{d}t
\right), \quad z\notin \mathcal{L}
$$
and
$$
\phi (W(z)):= \frac{-1}{2i\pi} \int_{\mathcal{L}_d} \frac{g(t)}{X^+ (W(t))} \frac{W'(t) }{(W(t)-W(z))} \, \mathrm{d}t, \quad z\notin \mathcal{L},
$$
where we denote $\mathcal{L}_d=(W^{-1})^- (\mathcal{M})$ (the red curve on the left picture of Figure \ref{C:fig1}).
We obtain the following proposition.
\begin{prop}[Solution of Carleman BVP]
The solutions of the Carleman BVP of Definition~\ref{C:def:carleman} are given by
\begin{equation} 
\Phi(z)=X(W(z))\phi(W(z))+X(W(z))P_{\chi}(W(z))
\end{equation}
where $P_{\chi}$ is a polynomial of degree $\chi$ for $\chi\geqslant 0$ and where $P_{\chi}=0$ for $\chi\leqslant -1$. For $\chi< -1$ the solution to the non-homogeneous problem exists if and only if some solubility conditions of the form \eqref{eq:solubility_condition} are satisfied. 
\label{C:prop_solution_carleman} 
\end{prop}

%

\section{Green's functions in dimension one}
\label{appendix:dim1}

This appendix is intended to be an educational approach that illustrates in a simple case the analytical method and the link between Green's functions and partial differential equations.
In this section we study $X$ a Brownian motion (in dimension one) with drift reflected at $0$. We are looking for Green's functions of $X$. This problem has already been studied in \cite{Chen_basic_1999}. Here we solve this question thanks to an analytic study which is much simpler than in dimension two. 

\begin{deff}[Reflected Brownian motion with drift]
\label{intro:def:MBreflechiDim1}
We define $X$, a reflected Brownian motion of variance $\rems{\sigma^2} $, of drift $\mu$ and starting from $x_0\in\mathbb{R}_+$, as the semi-martingale satisfying the equation
$$
{\rems{X(t)}}=x_0 + \rems{\sigma} W\rems{(t)}+\mu t + {\rems{L(t)}},
$$
where ${\rems{L(t)}}$ 
is the (symmetric) local time in $0$ of ${\rems{X(t)}}$ and $W_t$ is a standard Brownian motion.
\end{deff} 

\begin{deff}[Green measures]
Let $A\subset \mathbb{R}$ be a measurable set. \remst{The }Green's measure of the process $X$ starting from $x_0$ is defined by
$$
\rems{G({x_0},A)} = \mathbb{E}_{x_0} \left( \int_0^\infty \mathds1_A ({\rems{X(s)}}) \, \mathrm{d} s \right)
=  \int_0^\infty \mathbb{P}_{x_0} ({\rems{X(s)}} \in A ) \, \mathrm{d} s.
$$
Its density \rems{with respect to} the Lebesgue measure is denoted $g(x_0,x)$ and is called \remst{the }Green's function. \remst{The }Green's function satisfies
$$
g(x_0,x) = \int_0^\infty p(t,x_0,x) \, \mathrm{d} t,
$$
where $p(t,x_0,x)$ is the transition density of the process $X$.
\end{deff}
If $\mu>0$, the process is transient. \rems{In this case, $
G({x_0},A)<\infty$ for bounded subset $A\subset \mathbb{R}_+$.}
Furthermore notice that if $f:\mathbb{R\to\mathbb{R}_+}$ is measurable, by Fubini's theorem we have
\begin{align*}
\int_{\mathbb{R}}  f(x) \ g(x_0,x) \, \mathrm{d} x
=\mathbb{E}_{x_0} \left[ \int_0^{\infty}  f({\rems{X(t)}}) \, \mathrm{d} t \right] .
\end{align*}

\begin{prop}[Green's functions and Laplace transform]
If $\mu>0$, for all $x\in\mathbb{R}_+$ \remst{the }Green's function of $X$ is equal to
\begin{equation}
\label{eq:green_func_dim1}
g(x_0,x) = \frac{1}{\mu}e^{\frac{2\mu}{\rems{\sigma^2}} (x-x_0)} \mathbf{1}_{ \{0 \leqslant x < x_0\} } +
\frac{1}{\mu} \mathbf{1}_{ \{ x_0 \leqslant x \} } 
\end{equation}
and its Laplace transform $\psi^{x_0}$ is equal to
$$\psi^{x_0} (\theta) :=\int_0^\infty e^{\theta x} g(x_0,x) \, \mathrm{d}x
= - \frac{e^{\theta x_0} +\theta\frac{\rems{\sigma^2}}{2\mu} e^{-\frac{2\mu}{\rems{\sigma^2}} x_0}}{\mu \theta+\frac{1}{2}\rems{\sigma^2} \theta^2}.$$
\end{prop}
\begin{proof}
As in the two dimensional case, we are going to determine the Laplace transform of Green's function thanks to a functional equation.
If $f$ is a function $\mathcal{C}^2$, Itô formula gives
\begin{align*}
f({\rems{X(t)}}) &= f(x_0) +  \int_0^t f' ({\rems{X(s)}}) 
\, \mathrm{d} {\rems{X(t)}} 
+ \frac{1}{2} \int_0^t f'' ({\rems{X(s)}}) \, \mathrm{d}  \langle X \rangle_s  
\\
&= f(x_0) +  \int_0^t f' ({\rems{X(s)}}) 
( \mathrm{d} W(t) + \mu \mathrm{d} t + \mathrm{d} {\rems{L(t)}} )
+ \frac{1}{2} \int_0^t f'' ({\rems{X(s)}}) \sigma \, \mathrm{d} s  .
\end{align*}
For $f(x)= e^{\theta x}$ and $\theta <0$ we take the expectation of this formula and we obtain
$$
\mathbb{E}_{x_0} [e^{ \theta {\rems{X(t)}}}] = e^{\theta x_0} + \theta
\underbrace{
\mathbb{E}_{x_0} \left[\int_0^t e^{\theta {\rems{X(s)}}} \,  \mathrm{d} W(s) \right] 
}_{=0 \text{ because it is} \atop \text{the expectation of a martingale} }
+ (\mu \theta+\frac{1}{2}\rems{\sigma^2} \theta^2)  \mathbb{E}_{x_0} \left[\int_0^t e^{\theta {\rems{X(s)}}} \, \mathrm{d} s \right] + \theta  \mathbb{E}_{x_0} \left[\int_0^t e^{\theta {\rems{X(s)}}} \,  \mathrm{d} {\rems{L(s)}} \right].
$$
As $\theta <0$ and as ${\rems{X(t)}} \underset{t \to \infty}{\longrightarrow} \infty$ (as $\mu>0$), we have $\underset{t \to \infty}{\lim} \mathbb{E}[e^{ \theta {\rems{X(t)}}}] =0$. Let $t$ \rems{tend} to infinity. We obtain
\begin{align*}
0 &= e^{\theta x_0} +
(\mu \theta+\frac{1}{2}\rems{\sigma^2} \theta^2)  \mathbb{E}_{x_0} \left[\int_0^\infty e^{\theta {\rems{X(s)}}} \, \mathrm{d} s \right] + \theta  \mathbb{E}_{x_0} \left[\int_0^\infty e^{\theta {\rems{X(s)}}} \, \mathrm{d} {\rems{L(s)}} \right]
\\
&= e^{\theta x_0} + (\mu \theta+\frac{1}{2}\rems{\sigma^2} \theta^2)  \psi^{x_0} (\theta)
+ \theta  \mathbb{E} L({\infty})
\end{align*}
as $e^{\theta {\rems{X(s)}}}=1$ on the support of $\mathrm{d} {\rems{L(s)}}$ which is the set $\{s \geqslant 0: {\rems{X(s)}} =0 \}$. 
By evaluating at $\theta^* = -2\mu / \rems{\sigma^2}$ we find that $\mathbb{E} L({\infty}) = -  \frac{e^{\theta^* x_0}}{\theta^*} = \frac{\rems{\sigma^2} }{2\mu}e^{-\frac{2\mu}{\rems{\sigma^2}} x_0}$.
We obtain
$$
\psi^{x_0} (\theta) = - \frac{e^{\theta x_0} +\theta\frac{\rems{\sigma^2}}{2\mu} e^{-\frac{2\mu}{\rems{\sigma^2}} x_0}}{\mu \theta+\frac{1}{2}\rems{\sigma^2}\theta^2}.
$$
Inverting this Laplace transform we find formula \eqref{eq:green_func_dim1}.
\end{proof}
\begin{rem}[Partial differential equation]
Is is easy to verify that $g(x_0,x)$ satisfies the following partial differential equation
\begin{equation}
\label{eq:PDEdim1}
\begin{cases}
\frac{\rems{\sigma^2}}{2}\frac{\partial^2}{\partial x^2} g(x_0,x)
-\mu  \frac{\partial}{\partial x} g(x_0,x)
 = -\delta_{x_0} (x), & \
 \\
 \rems{\sigma^2}\frac{\partial}{\partial x} g(x_0,0) -2\mu g(x_0,0) =0, & \
 \end{cases}
\end{equation}
which is similar to equation \eqref{eq:robineq} in dimension two.
\end{rem}

\section{\rems{Generalization to a non-positive drift}}
\label{appendix:generalization}

\rems{In this paper, results are obtained for a positive drift: $\mu_1>0$ and $\mu_2>0$. In this appendix, we explain how to generalize these results to transient cases with a non-positive drift, that is when $\mu_1\leqslant 0$ or $\mu_2\leqslant 0$.
First of all, in these cases the ellipse $\gamma=0$ is oriented differently, see Figure~\ref{fig:ellipsegeneralize}. 
\begin{figure}[hbtp]
\centering
\includegraphics[scale=0.6]{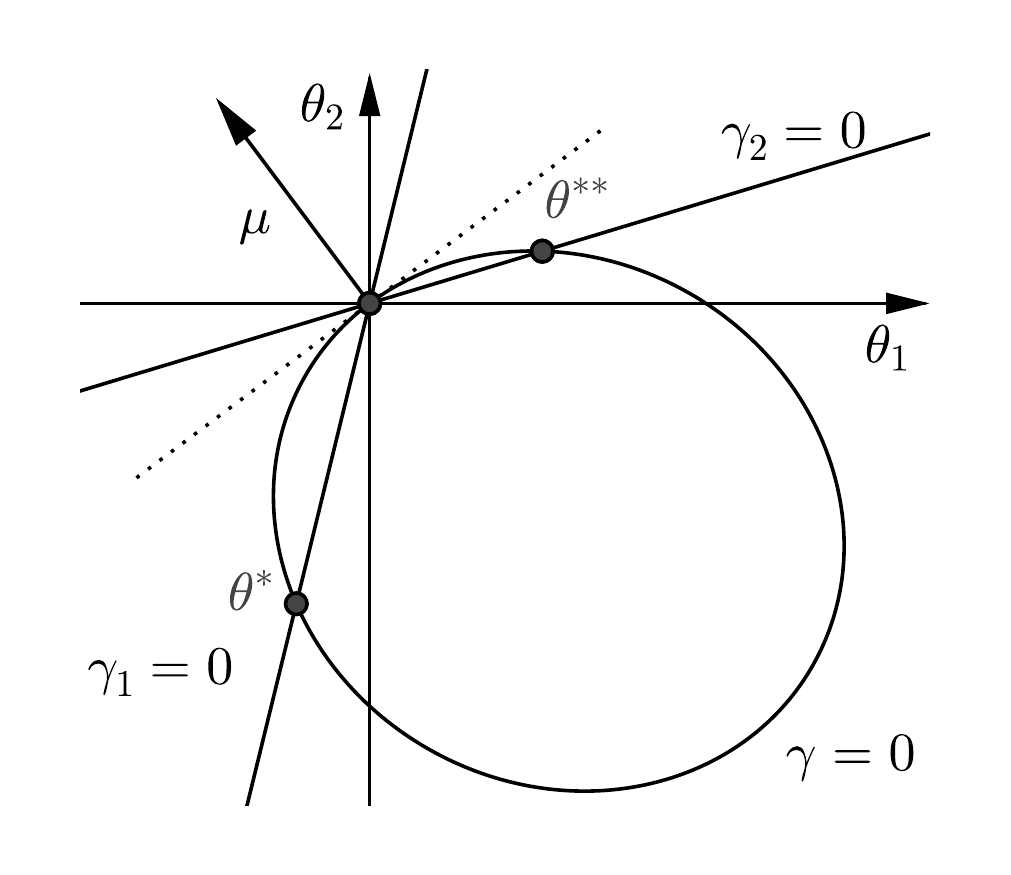}
\includegraphics[scale=0.6]{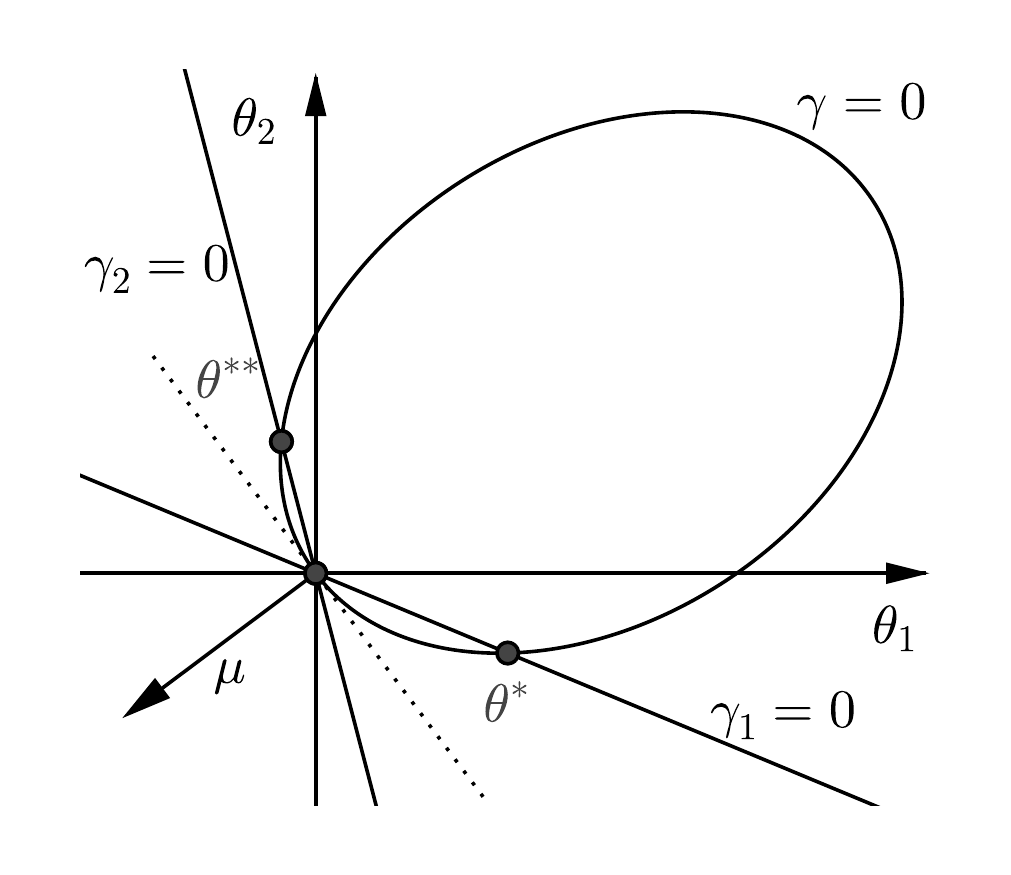}
\caption{On the left $\mu_1<0$ and $\mu_2>0$, on the right $\mu_1<0$ and $\mu_2<0$}
\label{fig:ellipsegeneralize}
\end{figure}
This leads to another set of convergence for the moment generating function. This is the main difference with the case of a positive drift. Analogously to Proposition~\ref{propeqfoncgreen}, we can show that
\begin{itemize}
\item 
when $\mu_1>0$ and $\mu_2\leqslant 0$: 
\begin{itemize}
\item $\psi_1(\theta_2) $ is finite on $\{\theta_2\in \mathbb{C} : \Re\theta_2\leqslant \theta_2^{**}  \wedge 0 \}
$,
\item $\psi_2(\theta_1)$ is finite on $\{\theta_1\in \mathbb{C} : \Re\theta_1 < 0  \}
$, 
\item $\psi(\theta)$ is finite on $\{\theta\in \mathbb{C}^2 : \Re\theta_1 < 0 \text{ and } \Re\theta_2\leqslant \theta_2^{**}  \wedge 0 \}
$;
\end{itemize}
\item when $\mu_1 \leqslant 0$ and $\mu_2>0$: \begin{itemize}
\item $\psi_1(\theta_2) $ is finite on $\{\theta_2\in \mathbb{C} :  \Re\theta_2<0 \}
$,
\item $\psi_2(\theta_1)$ is finite on $\{\theta_1\in \mathbb{C} : \Re\theta_1 \leqslant  \theta_1^{*}  \wedge 0  \}
$, 
\item $\psi(\theta)$ is finite on $\{\theta\in \mathbb{C}^2 : \Re\theta_2 < 0 \text{ and } \Re\theta_1 \leqslant  \theta_1^{*}  \wedge 0  \}
$;
\end{itemize}
\item when $\mu_1<0$ and $\mu_2<0$: 
\begin{itemize}
\item $\psi_1(\theta_2) $ is finite on $\{\theta_2\in \mathbb{C} : \Re\theta_2< 0   \}
$,
\item
$\psi_2(\theta_1)$ is finite on $\{\theta_1\in \mathbb{C} : \Re\theta_1 < 0  \}
$, 
\item $\psi(\theta)$ is finite on $\{\theta\in \mathbb{C}^2 : \Re\theta_1 < 0 \text{ and } \Re\theta_2< 0  \}
$.
\end{itemize}
\end{itemize}
In these sets the same functional equation \eqref{eq:functional_eq_green} still holds. As in Lemmas~\ref{lem:continuation_BM} and \ref{lem:domainincluded} but with some small technical differences in the proofs, it is then possible to continue the function $\psi_1$. We can therefore establish the same BVP as in Lemma~\ref{lem:BVP}. The resolution of this BVP is similar and leads to the same formula as \eqref{eq:main}. This generalization is the same phenomenon explained in \cite[\S 3.6]{franceschi_explicit_2017}.}

\bibliographystyle{apalike} 

\bibliography{biblio2,biblio1}

\end{document}